\documentclass[reqno,a4paper,11pt]{amsart}
\usepackage[T1]{fontenc}
\usepackage{mathrsfs}
\usepackage{libertine}
\usepackage[libertine]{newtxmath}
\usepackage[utf8]{inputenc}
\usepackage[english]{babel}
\usepackage{url}
\usepackage[%backref=page,
colorlinks=true, linkcolor = black, urlcolor=black, citecolor=blue, anchorcolor=blue]{hyperref}
\usepackage{comment}
\usepackage{tcolorbox}

\usepackage{cancel}

\newcommand{\crosses}[1]{%
	\ifcase#1\relax
	\or
	\rslash\or
	\rslash\mskip-5.5mu\rslash\or
	\rslash\mskip-5.5mu\rslash\mskip-5.5mu\rslash%
	\fi
}
\newcommand{\rslash}{\raisebox{.15ex}{/}}
\usepackage{xcolor,color}
\usepackage{geometry}
\usepackage[all]{xy}
\usepackage{enumitem}
\usepackage{booktabs}
\usepackage{slashed}
\usepackage{bbm}
\usepackage{cancel}

\usepackage{cleveref}
\usepackage{tikz}
\usepackage{tikz-cd}
\tikzcdset{arrow style=tikz, diagrams={>=stealth}}

\usepackage{xargs}  % Use more than one optional parameter in a new commands
\usepackage{soul}

\usepackage{amsmath}
\usepackage{amssymb,graphicx}
\usepackage{amsthm}
\usepackage{latexsym}
\usepackage{amsfonts}
\usepackage{bbm}
\usepackage{mathrsfs}
\usepackage{cases}

\usepackage{etex}
\calclayout

\numberwithin{equation}{section}

\theoremstyle{plain}
\newtheorem{lemma}{Lemma}[section]
\newtheorem{proposition}[lemma]{Proposition}
\newtheorem{proposition/definition}[lemma]{Proposition/Definition}
\newtheorem{theorem}[lemma]{Theorem}

\theoremstyle{definition}
\newtheorem{definition}[lemma]{Definition}
\newtheorem{remark}[lemma]{Remark}
\newtheorem{example}[lemma]{Example}

\DeclareRobustCommand{\SkipTocEntry}[5]{}

\newcommand{\Mod}{\ \text{mod}\ }

\newcommand{\calE}{\mathcal{E}}

\newcommand{\bbC}{\mathbb{C}}

\newcommand{\bbI}{\mathbb{I}}

\newcommand{\bbO}{\mathbb{O}}

\newcommand{\bbR}{\mathbb{R}}

\newcommand{\bbZ}{\mathbb{Z}}

\renewcommand{\theta}{\vartheta}

\newcommand{\Ltilde}{\widetilde{L}}
\newcommand{\Utilde}{\widetilde{L}_U}

\allowdisplaybreaks

\title{Homogeneous G-structures}

\makeatletter
\def\author@andify{%
	\nxandlist {\unskip ,\penalty-1 \space\ignorespaces}%
	{\unskip {} \@@and~}%
	{\unskip \penalty-2 \space \@@and~}%
}
\makeatother

\author[A. G. Tortorella]{Alfonso Giuseppe Tortorella}
\address{Department of Mathematics, KU Leuven, Celestijnenlaan 200B - 3001 Leuven, Belgium.}
\email{\href{mailto:alfonsogiuseppe.tortorella@kuleuven.be}{alfonsogiuseppe.tortorella@kuleuven.be}}

\author[L. Vitagliano]{Luca Vitagliano}
\address{DipMat, Università degli Studi di Salerno, via Giovanni Paolo II n◦ 123, 84084 Fisciano (SA), Italy.}
\email{\href{mailto:lvitagliano@unisa.it}{lvitagliano@unisa.it}}

\author[O. Yudilevich]{Ori Yudilevich}
\address{Department of Mathematics, KU Leuven, Celestijnenlaan 200B - 3001 Leuven, Belgium.}
\email{\href{mailto:ori.yudilevich@kuleuven.be}{ori.yudilevich@kuleuven.be}}

\keywords{}

\subjclass[2010]{53C10 (Primary), %G-structures,
				53D10 %Contact manifolds, general
			}

\begin{document}

\begin{abstract}
	The theory of $G$-structures provides us with a unified framework for a large class of geometric structures, including symplectic, complex and Riemannian structures, as well as foliations and many others. Surprisingly, contact geometry -- the ``odd-dimensional counterpart'' of symplectic geometry -- does not fit naturally into this picture. In this paper, we introduce the notion of a homogeneous $G$-structure, which encompasses contact structures, as well as some other interesting examples that appear in the literature. 
\end{abstract}

%\begin{abstract}
%\end{abstract}
%\date{\today}
\maketitle

\tableofcontents

\section{Introduction}

The theory of $G$-structures places a variety of geometric structures on equal footing, the idea being to encode a structure on a manifold $M$ by its set of compatible frames, which (in many interesting examples) forms a reduction of the frame bundle of the manifold to a structure group $G\subset \mathrm{GL}_n(\mathbb{R})$ (with $n = \dim M$). The group $G$ plays a key role in the theory, namely that of the linear model for the geometric structure. For example, a symplectic manifold induces a reduction of its frame bundle to the symplectic group, complex manifolds are modeled by the complex general linear group, Riemannian manifolds by the orthogonal group, volume forms by the special linear group, and so forth (see \cite{Kobayashi,Sternberg,Crainic} for introductions to the theory of $G$-structures).

The pattern that repeats itself in each example is as follows: every \textit{structure}, say one modeled by the group $G$, has a corresponding \textit{almost structure} where the integrability axiom is removed. The instances of the \textit{almost structure} that a given manifold admits are in one-to-one correspondence with reductions of the frame bundle of the manifold to $G$, and, of those, the instances of the \textit{structure} correspond to so-called \textit{integrable} reductions, which means that the manifold admits an atlas of coordinate charts that are compatible with the reduction (see Section \ref{section:homogeneousintegrability} for more details). For example, almost symplectic structures (i.e.~non-degenerate 2-forms) on a given manifold are in one-to-one correspondence with reductions of the frame bundle of the manifold to the symplectic group, and the symplectic structures (i.e.~\textit{closed} non-degenerate 2-forms) correspond to the integrable reductions. In this case, integrability is equivalent to the existence of an atlas consisting of Darboux charts. 

Contact structures, albeit being so similar to symplectic structures (most notably, due to the contact version of Darboux's theorem~\cite{Geiges}), do not fit into this picture. While the frame bundle of a (co-orientable) contact manifold can be reduced to the group $\mathrm U(n)\times 1$  (see~\cite[Ch. 5, Prop. 1.3]{Yano}), the reduction is not canonical, and, more problematically, the integrability axiom of the structure does not translate to the condition of the reduction being integrable as a $G$-structure\footnote{For more details, we recommend the mathoverflow discussion \href{https://mathoverflow.net/questions/281256/do-contact-and-cr-structures-have-corresponding-g-structures}{https://mathoverflow.net/questions/281256/do-contact-and-cr-structures-have-corresponding-g-structures}.}. In this paper, we provide a solution to this anomaly by introducing the notion of a \textit{homogeneous $G$-structure}. Let us illustrate the general idea by explaining what happens in the special case of contact structures. 

\subsection*{Contact Structures as Homogeneous Symplectic Structures}

A contact structure on a manifold $M$ is a corank-one distribution $H\subset TM$ which is maximally non-integrable (i.e. the curvature 2-form of $H$, rather than vanishing as in the integrable case, is non-degenerate). Let us write $L:=TM/H$ for the line bundle associated with $H$ and $\Ltilde :=L^*\backslash\{0\}$ for the complement of the zero section of the dual. The latter has the structure of a principal bundle when equipped with the obvious projection map $p:\Ltilde\to M$ and the action
\begin{equation}
\label{eqn:RactionMtilde}
h:\mathbb{R}^\times \times \Ltilde \to \Ltilde ,\hspace{1cm} (r,\epsilon)\mapsto h_r(\epsilon) = r\epsilon
\end{equation}
of the multiplicative group $\mathbb{R}^\times:=(\mathbb{R}\backslash\{0\},\cdot)$. 

A contact structure $H$ on $M$ induces a symplectic structure on $\Ltilde$ via a construction known as the ``symplectization''.  The symplectic form, which we denote by $\omega_H\in\Omega^2(\Ltilde)$, is obtained by pulling back the quotient map $TM\to TM/H = L$, viewed as an $L$-valued 1-form on $M$, to a usual 1-form on $\Ltilde$, and then applying the de Rham differential. Apart from being closed and non-degenerate, this 2-form also satisfies the homogeneity property
\begin{equation}\label{eq:HSF}
h_r^*\,\omega_H = r \omega_H,\hspace{1cm} \forall\, r\in\mathbb{R}^\times.
\end{equation}
Accordingly, we say that $\omega_H$ is \textit{homogeneous of degree 1}, since $r$ appears to the first power on the right hand side. (In the construction of $\Ltilde$, since the action \eqref{eqn:RactionMtilde} of $\mathbb{R}^\times$ restricts to an action of the positive reals $\mathbb{R}^\times_+$, one may wonder whether $\Ltilde$ divides into two connected components such that $\mathbb{R}^\times_+$ acts transitively on the fibers of each component, in which case it may be sufficient to take $\Ltilde$ to be one of these components in the story that follows. Obviously, in general, is not the case, because $TM/H$ may fail to be trivializable. A simple example of this is when $M=\mathbb{R}^k\times \mathbb{RP}^{k-1}$ and the contact structure is the canonical one constructed in Example (4) on p. 429 of \cite{Gray}).

Conversely, given a line bundle $L$ over $M$, any homogeneous of degree 1 symplectic form $\omega\in\Omega^2(\Ltilde)$ induces a contact structure on $M$ by contraction with the infinitesimal generator of the action $h$ (known as the Euler vector field). Indeed, by the homogeneity property, the resulting 1-form descends to an $L$-valued 1-form on $M$ whose kernel is a contact structure. Two such pairs $(L,\omega)$ and $(L',\omega')$ may induce the same contact structure on $M$, but, when they do, they are related by an equivalence, namely a vector bundle isomorphism $L \cong L'$ covering the identity map under which $\omega$ corresponds to $\omega'$ (auto equivalences are sometimes called conformal transformations).

The above constructions are inverse to one another, and, for a fixed manifold $M$, they define a one-to-one correspondence between contact structures $H$ on $M$, on the one hand, and pairs $(L,\omega)$ consisting of a line bundle $L$ over $M$ and a homogeneous of degree 1 symplectic structure $\omega$ on $\Ltilde$ modulo equivalence, on the other \cite{grabowski2015remarks}. 

\subsection*{Homogeneous $G$-Structures}

The ``symplectization'' hints at the idea of encoding a contact structure $H$ on $M$, with an associated line bundle $L = TM/H$, as a reduction of the frame bundle of $\Ltilde=L^*\backslash\{0\}$ to the symplectic group. However, to obtain a one-to-one correspondence as in the above examples of $G$-structures, we must be able to identify those reductions that are ``homogeneous of degree 1'', i.e. that come from a homogeneous of degree 1 symplectic structure. The approach we propose in this paper is to encode the homogeneity property of the symplectic form as the invariance of the corresponding reduction under a \textit{twisted} action of $\mathbb{R}^\times$ on the frame bundle of $\Ltilde$. The key, of course, is in the choice of the twisting. As we will see, the twisting is characterized by a map we call the \textit{degree}, a Lie group homomorphism of the form
\begin{equation*}
\alpha: \mathbb{R}^\times \to N(G)/G,
\end{equation*}
with $G$ the structure group, in this case the symplectic group, and $N(G)$ its normalizer inside the general linear group. In short -- the symplectic structure being homogeneous of degree 1 translates into the reduction being \textit{$\alpha$-homogeneous}, for an appropriate choice of $\alpha$. 

An advantage of this approach is that it can be generalized to other structure groups $G$ and other degree maps $\alpha$. Given any line bundle $L$ over a manifold $M$, any Lie subgroup $G\subset \mathrm{GL}_{n+1}(\mathbb{R})$ (with $n=\dim M$) and any map $\alpha$ as above, we will define the notion of an \textit{$\alpha$-homogeneous $G$-structure} on $\Ltilde$ (Definition \ref{def:homogeneousGstructure}). We will show that in addition to contact structures, the ``odd-dimensional counterparts'' of symplectic structures, our framework also encompasses the ``odd-dimensional counterparts'' of complex structures, a ``contact analogue'' of Riemannian metrics, and an example coming from Poisson geometry, or, more specifically, from structures known as $b$-symplectic manifolds (or also as log-symplectic manifolds).

\subsection*{A Terminological Remark} 

The careful reader has probably noticed already that, by a \emph{homogeneous} $G$-structure, say $S$, we do not at all mean that the group of symmetries of $S$ acts transitively on the underlying manifold. The word \emph{homogeneous} in the title is actually imported from the Poisson geometry literature, where a \emph{homogeneous symplectic form} is a symplectic form $\omega$ satisfying condition (\ref{eq:HSF}) with respect to an appropriate action $h$ of $\mathbb R^\times$.

\subsection*{Outline of the Paper}

The paper is organized as follows: in Section \ref{section:geometrylinebundle} we collect some facts about line bundles. In Section \ref{section:homogeneousGstructures}, we introduce the notion of a homogeneous $G$-structure, and in Section \ref{section:homogeneousintegrability}, the notion of \textit{homogeneous integrability}. In Section \ref{section:contactexample}, we prove that contact structures are in one-to-one correspondence with homogeneous $\mathrm{Sp}$-structures of an appropriate degree $\alpha$ that are homogeneous integrable, where $\mathrm{Sp}$ denotes the symplectic group, and we conclude in Section \ref{section:otherexamples} by proving analogous theorems for the three other examples mentioned above. 

\subsection*{Acknowledgments} 

A.~G.~Tortorella was supported by an FWO postdoctoral fellowship. L.~Vitagliano is a member of the GNSAGA of INdAM. O. Yudilevich was supported by the long term structural funding -- Methusalem grant of the Flemish Government, and by the FWO research project G083118N. The authors would also like to thank the Centre International de Recontre Math\'ematiques (CIRM) and its staff for their generous hospitality during our stay there as part of the Research in Pairs program, and the anonymous referee for his/her useful comments and suggestions.  

\section{Line Bundles}
\label{section:geometrylinebundle}

Let $M$ be a manifold and let $L$ be a line bundle over $M$. We use the standard notation that $C^\infty(M)$ denotes the ring of functions on $M$, $\mathfrak{X}(M)$ its $C^\infty(M)$-module of vector fields, and $\Gamma(L)$ the $C^\infty(M)$-module of sections of $L$, all in the smooth category. When working with contact structures and other examples of homogeneous $G$-structures, we will need to pass from ``usual'' geometry on $M$ to geometry on the line bundle $L$. The picture to keep in mind is the following:
\[ \begin{tikzcd}
\node (1-1) [anchor=east] {\text{Objects on $M$}}; & 
\node (1-2) [anchor=west] {\text{Atiyah objects on $L$}}; &
\node (1-3) [anchor=west] {\text{Homogeneous objects on $\Ltilde=L^*\backslash\{0\}$,}}; \\ 
\arrow[Rightarrow, dashed, from=1-1, to=1-2]
\arrow[leftrightarrow, from=1-2, to=1-3, above]{}{1-1}
\end{tikzcd}
\]
\vspace{-1cm}

\noindent where ``objects'' refers to the basic building blocks -- functions, vector fields, differential forms, etc. Let us explain this in slightly more detail (and we refer the reader to Section 2 of \cite{VitaglianoWade} for further details).

The role that functions on $M$ have in usual geometry is played by sections of $L$ (``Atiyah functions''). These, in turn, are in one-to-one correspondence with homogeneous of degree 1 functions on $\Ltilde$ (i.e.~functions $f\in C^\infty(\Ltilde)$ satisfying $h_r^*f=rf$ for all $r\in\mathbb{R}^\times$) via the correspondence $\lambda\in\Gamma(L)\mapsto \widetilde{\lambda}\in C^\infty(\Ltilde)$, where $\widetilde{\lambda}(\epsilon) := \epsilon(\lambda(p(\epsilon)))$. 

Vector fields on $M$ are replaced by derivations of $\Gamma(L)$ (``Atiyah vector fields''), i.e.~linear maps
\begin{equation*}
\Delta:\Gamma(L)\to \Gamma(L)
\end{equation*}
for which there exists a (necessarily unique) vector field $X_\Delta\in\mathfrak{X}(M)$ (the symbol of $\Delta$) such that
\begin{equation*}
\Delta(f\lambda) = f\Delta(\lambda) + X_\Delta(f) \lambda,\hspace{1cm} \forall\,\lambda\in\Gamma(L),\;f\in C^\infty(M).
\end{equation*} 
These are in one-to-one correspondence with homogeneous of degree 0 vector fields on $\Ltilde$ (i.e.~vector fields $X\in\mathfrak{X}(\Ltilde)$ satisfying $(h_r)_*X=X$ for all $r\in\mathbb{R}^\times$) via the correspondence $\Delta \mapsto \widetilde{\Delta}\in\mathfrak{X}(\Ltilde)$, where $\widetilde{\Delta}(\widetilde{\lambda}) := \widetilde{\Delta(\lambda)}$, for all $\lambda \in \Gamma (L)$. Under this correspondence, the identity operator $\bbI : \Gamma (L) \to \Gamma (L)$ corresponds to the infinitesimal generator $\mathcal E$ of the action $h$, namely the restriction to $\Ltilde$ of the Euler vector field on $L^\ast$.

A useful point of view to take is that derivations of $\Gamma(L)$ can be realized as the sections of a Lie algebroid $DL$ over $M$, called the \textit{Atiyah algebroid} of $L$ (see Example 3.3.4 in \cite{Mackenzie} or Section 2 of \cite{Vitagliano}). This is the Lie algebroid whose fiber at $x\in M$ consists of all pointwise derivations at $x$, i.e.~linear maps $\Delta_x:\Gamma(L)\to L_x$ for which there exists a (necessarily unique) vector $X_{\Delta_x}\in T_xM$ such that $\Delta_x(f\lambda) = f(x)\Delta_x(\lambda) + X_{\Delta_x}(f) \lambda_x$ for all $\lambda\in\Gamma(L)$ and $f\in C^\infty(M)$. Its bracket is the commutator bracket (of derivations) and its anchor is the symbol map $DL\to TM,\; \Delta_x\mapsto X_{\Delta_x}$. 

Going back to the picture above, one should think that the tangent bundle of $M$ is replaced by the Atiyah algebroid of $DL$. This allows us to complete the picture by replacing differential forms $\Omega^\bullet(M)$ on $M$ by differential forms on the Atiyah algebroid $DL$ with values in $L$ (``Atiyah forms''):
\begin{equation*}
\Omega^\bullet(DL;L):= \Gamma(\Lambda^\bullet (DL^*)\otimes L).
\end{equation*}
These, in turn, are in one-to-one correspondence with homogeneous of degree 1 differential forms on $\Ltilde$ (i.e. differential forms $\Omega \in \Omega^\bullet (\Ltilde)$ satisfying $(h_r)^\ast \Omega = r\Omega$ for all $r \in \bbR^\times$) via the correspondence $\omega \in \Omega^k (DL; L) \mapsto \widetilde \omega \in \Omega^k (\Ltilde)$, where $\widetilde \omega (\widetilde \Delta_1, \ldots, \widetilde \Delta_k) = \widetilde{\omega(\Delta_1, \ldots, \Delta_k)}$, for all $\Delta_1, \ldots, \Delta_k \in \Gamma (DL)$. Moreover, the de Rham differential $d:\Omega^\bullet(M)\to \Omega^{\bullet+1}(M)$ is replaced by the Lie algebroid differential $d_D:\Omega^\bullet(DL;L)\to \Omega^{\bullet+1}(DL;L)$ with values in the tautological representation $\Gamma(DL)\times \Gamma(L) \to \Gamma(L),\; (\Delta,\lambda) \mapsto \Delta(\lambda)$, which, under the one-to-one correspondence, is mapped to the usual de Rham differential $d:\Omega^\bullet(\Ltilde)\to \Omega^{\bullet+1}(\Ltilde)$. Note that the Atiyah complex $(\Omega^\bullet(DL;L),d_D)$ is acyclic (any closed form is exact), with $i_\bbI: \Omega^\bullet(DL;L) \to \Omega^{\bullet-1}(DL;L)$, insertion of the identity operator, acting as a homotopy operator. Also note that, since $\mathbb I$ spans the kernel of the symbol map $DL \to TM$, every Atiyah form of the type $i_\bbI \Omega$, with $\Omega \in \Omega^{k+1} (DL;L)$, descends to a unique $L$-valued form on $M$, $\vartheta \in \Omega^k (M;L)$, defined by $\vartheta (X_{\Delta_1}, \ldots, X_{\Delta_k}) = i_{\bbI} \Omega (\Delta_1, \ldots, \Delta_k)$, for all $\Delta_1, \ldots, \Delta_k \in \Gamma (DL)$.

\begin{remark}\label{rem:phi-hom}
In addition to homogeneous of degree $1$ functions on $\Ltilde$, we could also consider homogeneous functions of different degrees. In fact, for any Lie group homomorphism $\phi : \bbR^\times \to \bbR^\times$, we can consider \textbf{$\phi$-homogeneous functions} on $\Ltilde$, i.e. functions $f \in C^\infty (\Ltilde)$ such that 
$h_r^\ast f = \phi(r) f$ for all $r \in \bbR^\times$. Writing $p : \phi (L) \to M$ for the associated line bundle constructed out of the principal bundle $\Ltilde$ and the Lie group homomorphism $\phi$ (seen as a representation of the structure group $\bbR^\times$ on $\bbR$), sections of $\phi (L)$ are in one-to-one correspondence with $\phi$-homogeneous functions via the correspondence $\lambda \in \Gamma (\phi(L)) \mapsto \widetilde \lambda \in C^\infty (\Ltilde)$, with $\widetilde \lambda (\epsilon) := s$, where $s \in \bbR$ the unique real number such that $\lambda (p (\epsilon)) = [(\epsilon, s)]$. When $\phi$ is the trivial homomorphism, i.e. $\phi (r) = 1$ for all $r \in \bbR^\times$, then $\phi (L)$ is the trivial line bundle $\bbR_M = M \times \bbR \to M$, and $\phi$-homogeneous functions are homogeneous of degree $0$ functions (functions $f\in C^\infty (\Ltilde)$ such that $h_r^\ast f = f$ for all $r \in \bbR^\times$
). These, of course, are simply functions on the base $M$. When $\phi$ is the identity, then $\phi(L) = L$ and $\phi$-homogeneous functions are homogeneous of degree $1$ functions.

We also note that when $\phi : \bbR^\times \to \bbR^\times$ is a Lie group homomorphism with a non-trivial associated Lie algebra homomorphism (i.e.~$\phi$ is not locally constant), derivations of $\phi (L)$ are again in one-to-one correspondence with homogeneous of degree $0$ vector fields on $\Ltilde$ via the correspondence $\Delta \mapsto \widetilde \Delta$  given by the same formula as above, $\widetilde \Delta (\widetilde \lambda) = \widetilde{\Delta (\lambda)}$ for all $\lambda \in \Gamma (\phi (L))$. It follows that derivations of $\phi (L)$ are also in one-to-one correspondence with derivations of $L$, and hence this correspondence establishes a canonical Lie algebroid isomorphism $D\phi (L) \cong DL$ (beware that this works only when $\phi$ is not locally constant).
\end{remark}

%\begin{comment}
%
%When needed, we also denote elements of $\mathrm{Fr}(\Ltilde)$ by $(\epsilon,\psi)$. By differentiation, the action~\eqref{eqn:RactionMtilde} lifts to an action
%\begin{equation}
%\label{eqn:RactiononFr}
%h: \mathbb{R}^\times\times \mathrm{Fr}(\Ltilde)\to \mathrm{Fr}(\Ltilde),\hspace{1cm} h_r(\epsilon,\psi) = (r\epsilon,(h_r)_*\circ\psi).
%\end{equation}
%By the correspondence between degree 0 vector fields on $\Ltilde$ and section of $DL$, we see that the quotient of $\mathrm{Fr}(\Ltilde)$ by the action of $\mathbb{R}^\times$ is precisely the frame bundle of $DL$,
%\begin{equation*}
%\mathrm{Fr}(DL) = \{\; \psi : \mathbb{R}^{n+1}\xrightarrow{\cong} (DL)_x \; | \; x\in M \;\},
%\end{equation*}
%which is a principal $\mathrm{GL}_{n+1}(\mathbb{R})$-bundle with the right action by composition as above. Putting this all together, we obtain the following commutative diagram, in which every map has the structure of a principal bundle and the two actions on $\mathrm{Fr}(\Ltilde)$ commute:
%\[ \begin{tikzcd}
%\node (1-1) [anchor=east] {\mathrm{Fr}(\Ltilde) \hspace{-0.1cm} }; & 
%\node (1-2) [anchor=west] {\mathrm{Fr}(DL)}; \\ 
%\node (2-1) [anchor=east] {\Ltilde}; & 
%\node (2-2) [anchor=west] {M}; \\ [-23pt]
%\arrow[from=1-1, to=1-2]
%\arrow[from=2-1, to=2-2]
%\arrow[from=1-1, to=2-1, shift right = 0.35cm, start anchor=south east, end anchor=north east]
%\arrow[from=1-2, to=2-2, shift left = 0.35cm, start anchor=south west, end anchor=north west]
%\end{tikzcd}
%\]
%
%\end{comment}

\section{Homogeneous $G$-Structures}
\label{section:homogeneousGstructures}

In this section, we introduce the notion of a \textit{homogeneous $G$-structure}. Recall first that a $G$-structure on an $n$-dimensional manifold $M$, with $G\subset \mathrm{GL}_n(\mathbb{R})$ a Lie subgroup, is a reduction of the frame bundle of $M$,
\begin{equation*}
\mathrm{Fr}(M) = \{\; \psi:\mathbb{R}^n \xrightarrow{\cong} T_x M \; | \; x \in M \;\},
\end{equation*}
to the group $G$. Spelled out, it is a submanifold $S\subset \mathrm{Fr}(M)$ that: 1) is invariant under the restriction of the right action of $\mathrm{GL}_n(\mathbb{R})$,
\begin{equation*}
\mathrm{Fr}(M) \times \mathrm{GL}_n(\mathbb{R}) \to \mathrm{Fr}(M),\hspace{1cm} (\psi, g) \mapsto \psi \circ g,
\end{equation*}
to the subgroup $G$, and 2) has the structure of a principal $G$-bundle over $M$ when equipped with the restrictions of the action and the projection.

%\begin{comment}
%When $G$ is a closed subgroup, the quotient $\mathrm{Fr}(M)/G$ is smooth, and a $G$-structure on $M$ is the same thing as a section
%\begin{center}
%	\begin{tikzcd}
%		\node (1-1) [anchor=west]  {\hspace{-0.35cm}\mathrm{Fr}(M)/G \hspace{-0.1cm} } ;  \\ 
%		\node (2-1) [anchor=west] {M.}; \\ 
%		\arrow[from=1-1, to=2-1, shift left = 0.38cm, start anchor=south west, end anchor=north west]
%		\arrow[from=2-1, to=1-1, "s" {font=\normalsize, right}, shift right = 0.75cm, start anchor=north west, end anchor=south west, bend right = 35]
%	\end{tikzcd}
%	\vspace{-0.9cm}
%\end{center}
%\end{comment}

Now, let $L$ be a line bundle over an $n$-dimensional manifold $M$, and recall that $\Ltilde:= L^*\backslash\{0\}$ and that $p:\Ltilde\to M$ denotes the projection. Given a section of the frame bundle of $\Ltilde$,
\begin{center}
	\begin{tikzcd}
		\node (1-1) [left = 0.5cm]  {\mathrm{Fr}(\Ltilde)} ;  
		\node (2-1) [right = 0.5cm] {\Ltilde,};
		\arrow[from=1-1, to=2-1]
		\arrow[from=2-1, to=1-1, "\sigma" {font=\normalsize, right, yshift = 0.2cm, xshift= -0.15cm}, yshift = -0.1cm,  start anchor = north west, end anchor = north east,  bend right = 35]
	\end{tikzcd}
	%\vspace{-0.9cm}
\end{center}
or, in short, a \textbf{frame of $\Ltilde$}, there exists a (necessarily unique and smooth) map
\begin{equation*}
A_\sigma:\mathbb{R}^\times\times\Ltilde\to\mathrm{GL}_{n+1}(\mathbb{R}),\hspace{1cm} (r,\epsilon)\mapsto A_\sigma(r,\epsilon),
\end{equation*}
that satisfies $\sigma(r\epsilon) = (h_r)_* \circ \sigma(\epsilon) \circ A_\sigma(r,\epsilon)^{-1}$, for all $\epsilon\in\Ltilde$ and  $r\in\mathbb{R}^\times$. Thus, $A_\sigma$ measures how $\sigma$ varies along the orbits of the action of $\mathbb{R}^\times$ on $\Ltilde$. 

\begin{definition}
	A frame $\sigma$ of $\Ltilde$ is \textbf{homogeneous} if $A_\sigma(r,\epsilon)$ is independent of $\epsilon$, i.e. if $A_\sigma$ descends to a map
	\begin{equation}
	\label{eqn:AinducedbySigma}
	A_\sigma:\mathbb{R}^\times\to \mathrm{GL}_{n+1}(\mathbb{R}).
	\end{equation}
\end{definition}

\begin{lemma}
	If a frame $\sigma$ of $\Ltilde$ is homogeneous, then $A_\sigma$ is a Lie group homomorphism, and it induces a left action
	\begin{equation}
	\label{eqn:twistedRactiononFr}
	\mathbb{R}^\times \times \mathrm{Fr}(\Ltilde) \to \mathrm{Fr}(\Ltilde),\hspace{1cm} (r,\psi) \mapsto r\cdot_{A_\sigma}\psi := (h_r)_* \circ \psi \circ A_\sigma(r)^{-1}.
	\end{equation}
\end{lemma}

\begin{proof}
	Let $r,s\in\mathbb{R}^\times$ and $\epsilon\in\Ltilde$. Since $\sigma(r\epsilon) = (h_r)_* \circ \sigma(\epsilon) \circ A_\sigma(r)^{-1}$, then
	\begin{equation*}
	\begin{split}
	\sigma(rs\epsilon) &= (h_{rs})_*\circ \sigma(\epsilon)\circ A_\sigma(rs)^{-1}, \\
	\sigma(rs\epsilon) &= (h_{r})_*\circ \sigma(s\epsilon)\circ A_\sigma(r)^{-1} = (h_{rs})_*\circ \sigma(\epsilon)\circ A_\sigma(s)^{-1}  A_\sigma(r)^{-1},
	\end{split}
	\end{equation*}
	and thus $A_\sigma(rs) = A_\sigma(r)A_\sigma(s)$. The second assertion is now straightforward. 
\end{proof}

\begin{remark}
	\label{rem:characterization_A's}
	Lie group homomorphisms of the type $A:\bbR^\times\to\mathrm{GL}_{n+1}(\bbR)$ are in one-to-one correspondence with pairs $(B,C)\in\mathfrak{gl}_{n+1}(\bbR)\times\mathrm{GL}_{n+1}(\bbR)$ that satisfy
	\begin{equation}
	\label{eq:rem:characterization_A's:constraint}
	C^2=\bbI\quad \text{and}\quad C\exp(Bt)=\exp(Bt)C, \hspace{1cm} \text{for all}\ t\in\bbR.
	\end{equation}
	In one direction, one sets $B:=\mathsf{Lie}(A)(1)\in\mathfrak{gl}_{n+1}(\bbR)$, where $\mathsf{Lie}(A):\bbR\to\mathfrak{gl}_{n+1}(\bbR)$ is the induced Lie algebra homomorphism, and $C:=A(-1)\in\mathrm{GL}_{n+1}(\bbR)$. Conversely, we recover $A$ by
	\begin{equation*}
	\label{eq:rem:characterization_A's:action}
	A(r)=\begin{cases}
	\exp(B\log(r))&\text{if}\ r>0,\\
	C\exp(B\log(|r|))&\text{if}\ r<0.
	\end{cases}
	\end{equation*}
\end{remark} 

While homogeneous frames (and, in general, frames) may fail to exist globally, they always exist locally on \textbf{saturated open subsets} of $\Ltilde$, i.e. open subsets of the type 
\begin{equation*}
\Utilde:= p^{-1}(U) \subset \Ltilde, \hspace{1cm} \text{with $U\subset M$ open,}
\end{equation*}
assuming that $U$ is sufficiently small. Indeed, for any $x\in M$, we may construct an open neighborhood $U\subset M$ such that there exists a section $\eta$ of $\Utilde\to U$ and a local section $\sigma_0$ of $\mathrm{Fr}(\Utilde)\to \Utilde$ defined on a neighborhood of $\eta(U)$. Then, given any Lie group homomorphism $A : \bbR^\times \to \mathrm{GL}_{n+1} (\bbR)$, we define a homogeneous frame $\sigma$ of $\Utilde$ with $A_\sigma=A$ by imposing invariance under the action \eqref{eqn:twistedRactiononFr}, i.e. by setting $\sigma (\epsilon) := (h_{r(\epsilon)})_\ast \circ \sigma_0 (\eta(p(\epsilon))) \circ A (r(\epsilon))^{-1}$ for all $\epsilon\in\Utilde$, where $r(\epsilon)\in \bbR^\times$ is determined by $\epsilon = r(\epsilon)\eta(p(\epsilon))$. We will use the term \textbf{semi-local homogeneous frame} of $\Ltilde$ around $\epsilon$ (or \textbf{semi-local homogeneous section} of $\mathrm{Fr}(\Ltilde)$ around $\epsilon$) for a homogeneous frame defined on a saturated open neighborhood of $\epsilon\in\Ltilde$. 

\begin{definition}
	\label{def:homogeneousGstructure}
	Let $G\subset \mathrm{GL}_{n+1}(\bbR)$ be a Lie subgroup. A \textbf{homogeneous $G$-structure} on $\Ltilde$ (with $\dim \Ltilde = n+1$) is a $G$-structure $S\subset\mathrm{Fr}(\Ltilde)$ such that, for any $\epsilon\in \Ltilde$, 
	\begin{enumerate}
		\item there exists a semi-local homogeneous section $\sigma$ of $S$ around $\epsilon$, say with domain $\Utilde\subset \Ltilde$,
		\item $S|_{\Utilde}$ is preserved by the action \eqref{eqn:twistedRactiononFr} induced by $A_\sigma$, i.e. 
		$r\cdot_{A_\sigma} S|_{\Utilde}=S|_{\Utilde}$ for all $r\in \mathbb{R}^\times$.
	\end{enumerate}
\end{definition}

Note that homogeneous $G$-structures can be restricted to saturated open subsets, namely if $S$ is a homogeneous $G$-structure on $\Ltilde$, then $S|_{\Utilde}$ is a homogeneous $G$-structure on $\Utilde$.

The second condition in the above definition has the following useful characterization:

\begin{lemma}\label{lem:N(G)}
	Let $S$ be a $G$-structure on $\Ltilde$ and let $\sigma$ be a homogeneous section of $S$ (i.e. a homogeneous frame of $\Ltilde$ with values in $S$). Then $r\cdot_{A_\sigma} S=S$ for all $r\in \mathbb{R}^\times$ if and only if $A_\sigma$ takes values in the normalizer $N(G)$ of $G$.
\end{lemma}

\begin{proof}
	Assume that $r\cdot_{A_\sigma} S=S$ for all $r\in \mathbb{R}^\times$. For all $\epsilon\in\Ltilde, \; g\in G$ and $r\in\mathbb{R}^\times$, the left hand sides of
	\begin{equation*}
	\begin{split}
	r\cdot_{A_\sigma}(\sigma(\epsilon)\circ g) &= (h_r)_*\circ \sigma(\epsilon)\circ g A_\sigma(r)^{-1}   \\ 
	\sigma(r\epsilon)\circ g' &= (h_r)_*\circ \sigma(\epsilon)\circ A_\sigma(r)^{-1} g',
	\end{split}
	\end{equation*}
	are equal for some $g'\in G$, and the equality on the right hand sides then implies that $A_\sigma(r)gA_\sigma(r)^{-1} = g'$, and hence that $A_\sigma$ takes values in $N(G)$. Conversely, assume that $A_\sigma$ takes values in $N(G)$. Any $\psi\in S$ can be written as $\sigma(\epsilon)\cdot g$ for some $\epsilon\in\Ltilde$ and $g\in G$. Given $r\in\mathbb{R}^\times$, the two right hand sides above are equal for some $g'\in G$ due to the assumption, and so the equality on the left hand sides shows that $r\cdot_{A_\sigma}\psi\in S$. 
\end{proof}

Of course, the above definition of a homogeneous $G$-structure should not depend on the choices of semi-local homogeneous sections:

\begin{proposition}\label{prop:sigma,sigma'}
	Let $S$ be a homogeneous $G$-structure on $\Ltilde$. If $\sigma$ is a homogeneous section of $S$ such that $A_\sigma$ takes values in $N(G)$, then, given any other homogeneous section $\sigma'$ of $S$,  $A_{\sigma'}$ takes values in $N(G)$. Furthermore, 
	\begin{equation*}
	A_\sigma \equiv A_{\sigma'} \Mod G,
	\end{equation*}
in the sense that the compositions of $A_\sigma$ and $A_{\sigma'}$ with the projection $N(G) \to N(G)/G$ are equal.
\end{proposition}

\begin{proof}
For all $\epsilon\in\Ltilde$,
	\begin{equation*}
	\sigma'(\epsilon) = \sigma(\epsilon)\circ g(\epsilon)
	\end{equation*}
	for some smooth function $g:\Ltilde\to G$. Since $\sigma$ and $\sigma'$ are both homogeneous, then for all $r\in\mathbb{R}^\times$,
	\begin{equation*}
	\begin{split}
	&\sigma'(r\epsilon) = (h_r)_*\circ \sigma'(\epsilon)\circ A_{\sigma'}(r)^{-1}, \\
	&\sigma(r\epsilon)\circ g(r\epsilon) = (h_r)_*\circ \sigma(\epsilon)\circ A_\sigma(r)^{-1}g(r\epsilon) = (h_r)_*\circ \sigma'(\epsilon)\circ g(\epsilon)^{-1} A_\sigma(r)^{-1}g(r\epsilon).
	\end{split}
	\end{equation*}
	Since the two left hand sides are equal, it follows that
	\begin{equation*}
	A_{\sigma'}(r) = g(r\epsilon)^{-1}A_\sigma(r)g(\epsilon),
	\end{equation*}
	and hence $A_{\sigma'}(r)$ takes values in $N(G)$ (since $G\subset N(G)$, and $N(G)$ is a group), and $A_\sigma(r)$ and $A_{\sigma'}(r)$ belong to the same coset of $G$ in $N(G)$ (since $G$ is normal in $N(G)$, and left and right cosets coincide). 
\end{proof}

A consequence of this proposition is that, given a homogeneous $G$-structure $S$ over $\Ltilde$, there is a canonical Lie group homomorphism 
\begin{equation}
\label{eqn:alpha}
\alpha:\mathbb{R}^\times \to N(G)/G,\hspace{1cm} r\mapsto A_\sigma(r)G,
\end{equation}
associated with every connected component $M_0$ of the base manifold $M$. Here, $\sigma$ is any choice of a homogeneous section of $S|_{\Utilde}$, with $U$ a sufficiently small, non-empty open subset in $M_0$. When this map is the same for all connected components, we call $\alpha$ the \textbf{degree} of $S$, and we say that $S$ is an \textbf{$\alpha$-homogeneous $G$-structure}. A \textbf{lift} of $\alpha$ is any Lie group homomorphism $A : \bbR^\times \to N(G)$ such that $\alpha$ is equal to the composition of $A$ with the projection $N(G)\to N(G)/G$.

\begin{proposition}
\label{prop:any_A}
Let $S$ be an $\alpha$-homogeneous $G$-structure on $\Ltilde$. Given $\epsilon \in \Ltilde$ and a lift $A : \bbR^\times \to N(G)$ of $\alpha$ (if one exists), there exists a semi-local homogeneous section $\sigma$ of $S$ around $\epsilon$ such that $A_\sigma = A$.
\end{proposition}

\begin{proof}
Start with any homogeneous section $\sigma_0$ of $S|_{\Utilde}$, with $\Utilde$ a sufficiently small saturated open neighborhood of $\epsilon$ such that there exists a section $\eta$ of the projection $\Utilde\to U$. Then set 
\begin{equation*}
\sigma (\epsilon) := \sigma_0 (\epsilon) \circ A_{\sigma_0}(r(\epsilon))A(r(\epsilon))^{-1},
\end{equation*}
where $r(\epsilon) \in \bbR^\times$ is determined by $\epsilon = r(\epsilon) \cdot \eta (p(\epsilon))$.
\end{proof}

\begin{remark}
	Proposition \ref{prop:any_A} gives an obstruction for the existence of a homogeneous $G$-structure with a prescribed degree $\alpha:\mathbb{R}^\times \to N(G)/G$, namely that $\alpha$ must admit a lift to a Lie group homomorphism $A:\mathbb{R}^\times \to N(G)$. For example, this obstruction is non-trivial when $M$ is a point (and so $n = 0$ and $\mathrm{GL}_{n+1}(\bbR) = \bbR^\times$), $G$ is the (closed) subgroup generated by $2$ (in which case $N(G) = \bbR^\times$), and $\alpha : \bbR^\times \to \bbR^\times/G$ is given by
	\[
	\alpha (r) =
	\left\{
	\begin{array}{cl}
	rG & \text{if $r > 0$} \\
	\sqrt{2}|r| G & \text{if $r < 0$.}
	\end{array}
	\right.
	\]
	Here, $\alpha(-1) = \sqrt{2}G$, and, since there is no order two element in the coset $\sqrt{2}G$, $\alpha$ cannot be lifted to a homomorphism $\bbR^\times \to \bbR^\times$. This counter-example can be easily generalized to higher dimensions. When $\alpha$ does admit a lift $A$, while there is still no guarantee for the global existence of an $\alpha$-homogeneous $G$-structure on a given $\Ltilde$. It is, however, sufficient for local existence, since we can construct a homogeneous frame $\sigma$ with $A_\sigma=A$ on a saturated open subset $\Utilde$ (as explained above), and extend it to the unique homogeneous $G$-structure on $\Utilde$ that contains the image of $\sigma$. 
\end{remark}

\section{Homogeneous Integrability}
\label{section:homogeneousintegrability}

We now move on to discuss integrability in the context of homogeneous $G$-structures. Recall that a $G$-structure $S\subset \mathrm{Fr}(M)$ on a manifold $M$ is said to be integrable if around every point in $M$ there exists a coordinate chart $(U,\chi)$ such that the induced frame
\[
\sigma_\chi := \left( \frac{\partial}{\partial \chi^1}, \ldots, \frac{\partial}{\partial \chi^n}\right),
\]
viewed as a local section of $\mathrm{Fr}(M)$, takes values in $S$. In the case of homogeneous $G$-structures, motivated by the examples that will be presented in the following two sections, we are interested in \textit{homogeneous coordinate charts}. As always, $L\to M$ is a line bundle and $\Ltilde = L^*\backslash \{0\}$. 

%\begin{comment}
%Integrability can also be characterized purely in terms of frames. Let $\sigma=(\sigma^1,...,\sigma^n)$ be a frame of $M$, with $\sigma^i\in \mathfrak{X}(M)$ its $i$'th component. We say that $\sigma$ is \textbf{commuting} if $[\sigma^i,\sigma^j]=0$ for all $i,j=1,...,n$. A $G$-structure $S$ on $M$ is integrable if and only if there exists a local commuting section of $S$ around every point of $M$. 
%\end{comment}

 \begin{definition}
 A coordinate chart $(V, \chi)$ of $\Ltilde$ is \textbf{homogeneous} if, locally around every point of $V$, the induced frame $\sigma_\chi$ is the restriction of a semi-local homogeneous frame of $\Ltilde$. A homogeneous $G$-structure $S$ on $\Ltilde$ is \textbf{homogeneous integrable} if around every point of $\Ltilde$ there exists a homogeneous coordinate chart such that $\sigma_\chi$ takes values in $S$.
 \end{definition}
 
 \begin{remark}
 In most of the examples that we have of homogeneous $G$-structures, homogeneous integrability is equivalent to integrability (see Theorems \ref{theor:cont}, \ref{theor:sympl_triv}, \ref{theor:complex}). However, the proof in each example is rather different and particular to that case. In the example of Section \ref{sec:orthogonal} we were only able to prove that integrability implies homogeneous integrability in a special case. Summarizing, we do not know if this fact is true in general.
 \end{remark}

The condition for a coordinate chart to be homogeneous can also be rephrased in the following more intrinsic way:

\begin{proposition}
	A coordinate chart $(V, \chi)$ of $\Ltilde$ is homogeneous if for every point $\epsilon_0 \in V$, there exist 
	\begin{enumerate}
		\item an open neighborhood $V_0$ of $\epsilon_0$ in $V$, 
		\item an open neighborhood $I_0$ of $1$ in $\bbR^\times$,
		\item a Lie group homomorphism $(A, b) : \bbR^\times \to \mathrm{Aff}_{n+1}(\bbR) = \mathrm{GL}_{n+1} (\bbR) \ltimes \bbR^{n+1}$,
	\end{enumerate}
	such that $h_r (V_0) \subset V$, and 
	\begin{equation}\label{eq:h*chi}
	h_r^\ast \chi = A(r)\chi + b(r) \quad \text{on $V_0$,}
	\end{equation}
	for all $r \in I_0$. Furthermore, if $\sigma$ is a semi-local homogeneous frame of $\Ltilde$ whose restriction to $V$ is $\sigma_\chi$, then $A_\sigma = A$ on the connected component of the identity $\bbR^\times_+\subset \mathbb{R}^\times$. 
\end{proposition}

\begin{proof} %The (easy) proof of this proposition is uninteresting for the scopes of the paper. For the sake of brevity we leave it to the reader.%\quad
	Begin with a homogeneous coordinate chart $(V, \chi)$ on $\Ltilde$, let $\epsilon_0 \in V$, and let $V_0$ be a connected open neighborhood of $\epsilon_0$ in $V$ such that $\sigma_\chi$
	agrees with a local homogeneous frame, say $\sigma$, in $V_0$. Shrinking $V_0$ if necessary, we can assume that $h_r (V_0) \subset V$ for all $r$ in a sufficiently small interval $I_0 \subset \bbR^\times$ containing $1$. It is easy to see that $A_\sigma (r)$ maps the $i$-th element in the canonical frame of $\bbR^{n+1}$ to
	\[
	\left( \frac{\partial h_r^\ast (\chi^1)}{\partial \chi^i}(\epsilon), \ldots, \frac{\partial h_r^\ast (\chi^{n+1})}{\partial \chi^i} (\epsilon)\right)
	\]
	for all $\epsilon \in V_0$, and $r \in I_0$. As $A_{\sigma}(r)$ is independent of $\epsilon$, we have
	\[
	\frac{\partial^2 h_r^\ast (\chi^{k})}{\partial \chi^j \partial \chi^i} = 0,
	\]
	for all $i,j, k = 1, \ldots, n+1$, hence
	\[
	h_r^\ast \chi = A_\sigma (r) \chi + b(r),
	\]
	for some $b(r) \in \bbR^{n+1}$. From the group property of $h_r$, $(A_\sigma, b)$ is a local Lie group homomorphism. As such, it can be uniquely extended to a Lie group homomorphism from the connected component of the identity $\bbR^\times_+$ of $\bbR^\times$ to $\mathrm{Aff}_{n+1} (\bbR)$. Finally, extend $(A_\sigma, b)$ arbitrarily to a Lie group homomorphism 
	\[
	(A, b) : \bbR^\times \to \mathrm{Aff}_{n+1} (\bbR).
	\]
	
	\bigskip

Conversely, let $(V, \chi)$ be a chart on $\Ltilde$ as in the statement. Put $U_0 = p (V_0)$, and $x_0 = p (\epsilon_0)$. Shrinking $V_0$ if necessary, we can assume that $p : V_0 \to U_0$ is a trivial fiber bundle. Let $\eta$ be a section of $V_0 \to U_0$ such that $\eta (x_0) = \epsilon_0$. As already remarked before, we can construct a semi-local homogeneous frame $\sigma$ on $\Ltilde_{U_0}$ by setting
	$\sigma (\epsilon) := (h_{r(\epsilon)})_\ast \circ \sigma_0 (\eta(p(\epsilon))) \circ A (r(\epsilon))^{-1}$ for all $\epsilon\in\Utilde$, where $r(\epsilon)\in \bbR^\times$ is determined by $\epsilon = r(\epsilon)\eta(p(\epsilon))$. By construction, $A_\sigma = A$. Finally, it easily follows from (\ref{eq:h*chi}), that $\sigma$ and
	$\sigma_\chi$
	agree on $V_0$.
%%	
%%	
%%	\bigskip
%%	
%%	For the last part of the statement it is enough to show that if $\sigma$ agrees with 
%%	$\sigma_\chi$
%%	around $\epsilon_0$, then it is commuting in a whole saturated neighborhood of $\epsilon_0$. So, for any $s \in \bbR^\times$, let $V_{(s)} = h_s (V_0)$ (so that $V_{(1)} = V_0$). From (\ref{eq:h*chi}), the pair $(V_{(s)}, \chi_{(s)} := h_{s^{-1}}^\ast \chi)$ is a homogeneous chart, and $\sigma_{\chi_{(s)}}$
%%	agrees with $\sigma$ on $V_{(s)}$. Hence $\sigma$ is a commuting on the whole $\cup_s V_{(s)} = p^{-1}(p(V_0))$.
\end{proof}
 
 Note that we do not require that the domain of a homogeneous coordinate chart be a saturated open subset as we do for semi-local homogeneous frames. The reason for this is that charts of this type cannot always be extended to saturated domains, as the following example shows:
 
 \begin{example}
 	Let $M = \mathbb R^n$ with the standard coordinates $(x^1, \ldots, x^n)$, and let $L$ be the trivial line bundle  $\bbR_M:= M \times \bbR \to M$, so that $\Ltilde = \mathbb R^n \times \bbR^\times$. Let $\mu$ be the standard coordinate on $\bbR^\times$, and consider the coordinate chart
 	\[
 	\left( V = \mathbb R^n \times \bbR^\times_+, \chi = ( x^1, \ldots, x^n, \log \mu ) \right)
 	\]
 	on $\Ltilde$. The induced coordinate frame is
 	\[
 	\sigma_\chi = \left(\frac{\partial}{\partial x^1}, \ldots, \frac{\partial}{\partial x^n}, \mu \frac{\partial}{\partial \mu}  \right).
 	\]
 	This frame extends (via the same formula) to a commuting homogeneous section of $\mathrm{Fr} (\Ltilde)$, with $A_\sigma$ the trivial homomorphism, while $(V,\chi)$ cannot be extended to the whole of $\Ltilde$ since $\chi$ is already surjective onto $\mathbb{R}^{n+1}$. 
 \end{example}

\section{Contact Structures as Homogeneous $G$-Structures}
\label{section:contactexample}

Our main and motivating example of a homogeneous $G$-structure is a contact structure. As explained in the introduction, the linear model for a contact structure on an odd dimensional manifold $M$, say with $n = \dim M$, is the symplectic group $\mathrm{Sp}_{k} \subset \mathrm{GL}_{2k}(\bbR)$, with $k = (n+1)/2$, consisting of $2k \times 2k$ matrices $A$ satisfying $A^t J A = J$. Here,
\begin{equation}\label{eq:i_matrix}
J = \left(
\begin{array}{cc}
\bbO & \bbI_k \\
-\bbI_k & \bbO
\end{array}
\right),
\end{equation}
with $\bbI_k$ the $k\times k$ unit matrix. 

\begin{lemma}[{\cite[Theorem 1.10]{Sirola}}]
\label{lem:N(Sp)}
%The normalizer $N(\mathrm{Sp}_{k})$ of the symplectic group $\mathrm{Sp}_{k}$ in spanned by $\mathrm{Sp}_{k}$ itself, (invertible) scalar matrices, and the matrix
%\[
%\left(
%\begin{array}{cc}
%\bbO & \bbI \\
%\bbI & \bbO
%\end{array}
%\right).
%\]
%The normalizer $N(\mathrm{Sp}_{k})$ of the symplectic group $\mathrm{Sp}_{k}$ in $\mathrm{GL}_{n+1}$ is the semidirect product of $\mathrm{Sp}_{k}$ itself and the $1$-dimensional subgroup $R \subset \mathrm{GL}_{n+1}$ generated by the matrices of the form
%\begin{equation}\label{matrix}
%\left(
%\begin{array}{cc}
%\bbI & \bbO \\
%\bbO & r\bbI
%\end{array}
%\right), \quad r \in \bbR^\times.
The normalizer $N(\mathrm{Sp}_{k})$ of the symplectic group $\mathrm{Sp}_{k}\subset\mathrm{GL}_{2k}(\bbR)$ fits in the split short exact sequence of Lie group homomorphisms
\begin{equation}
\label{eq:SES_Sp}
	\begin{tikzcd}
	1 \arrow[r]&\mathrm{Sp}_{k}\arrow[r]&N(\mathrm{Sp}_{k})\arrow[r, %swap,
	 "p"]&\mathbb{R}^\times\arrow[r]\arrow[l, bend right=-40, %swap, 
	 "A" pos=.55]&1,
	\end{tikzcd}
	\end{equation}
with $p:N(\mathrm{Sp}_{k})\to\bbR^\times$ defined by $p(B)\bbI_{n+1}=J^tB^tJB$. Furthermore, a splitting $A$ is given by
\begin{equation}
\label{eq:SES_Sp_splitting}
A:\bbR^\times\to N(\mathrm{Sp}_{k}),\hspace{1cm} 
	r\mapsto
\begin{pmatrix}
\bbI&\bbO\\
\bbO& r\bbI
\end{pmatrix},
\end{equation}
and therefore $N(\mathrm{Sp}_{k})$ decomposes as the semidirect product of $\mathrm{Sp}_{k}$ and the $1$-dimensional subgroup consisting of matrices of the form $A(r)$, with $r \in \bbR^\times$.
\end{lemma}

As a preparation for the theorem below, let us explain how a contact structure is constructed out of an $\alpha$-homogeneous $\mathrm{Sp}_{k}$-structure $S$ on $\Ltilde$, where the relevant degree in this case is simply the identity map $\alpha : \bbR^\times \to N(\mathrm{Sp}_{k})/\mathrm{Sp}_{k} \cong \bbR^\times$, where the last isomorphism is the one induced by \eqref{eq:SES_Sp}. By Proposition \ref{prop:any_A}, around any point in $\Ltilde$, there is a saturated open neighborhood $\Utilde$ and a homogeneous section $\sigma$ of $S|_{\Utilde}$ such that $A_\sigma = A$, where $A$ is the lift of $\alpha$ given by \eqref{eq:SES_Sp_splitting}. Due to the specific form of $A$, the components $X_1, \ldots, X_{k}, Y^1, \ldots, Y^{k} \in \mathfrak{X}(\Utilde)$ of $\sigma$ satisfy the homogeneity conditions
\begin{equation}
\label{eq:symp_frame_hom_cond}
(h_r)_\ast (X_i) = X_i \hspace{0.5cm} \text{and} \hspace{0.5cm} (h_r)_\ast (Y^i) = r Y^i \hspace{1cm} \forall \; r \in \bbR^\times,\;  i = 1,\ldots, k.
\end{equation}
Denoting the components of the dual coframe by $\xi^1, \ldots, \xi^{k}, \eta_1, \ldots, \eta_{k}\in\Omega^1(\Utilde)$, we define a non-degenerate $2$-form $\widetilde \omega \in \Omega^2(\Ltilde)$ on $\Ltilde$ by setting
\[
\widetilde \omega|_{\Utilde} = \xi^1 \wedge \eta_1 + \cdots + \xi^{k} \wedge \eta_{k}
\]
on the saturated open neighborhood $\Utilde$. Due to \eqref{eq:symp_frame_hom_cond}, $\widetilde \omega$ satisfies the homogeneity condition
\[
h^\ast_r \, \widetilde \omega = r \widetilde \omega, \quad \forall\; r \in \bbR^\times,
\]
which, as explained in Section \ref{section:geometrylinebundle}, implies that $\widetilde \omega$ uniquely determines a non-degenerate Atiyah $2$-form $\omega : \wedge^2 DL \to L$.
%Conversely, begin with an almost symplectic structure $\omega : \wedge^2 DL \to L$. Locally, around every point of $M$, we can choose
%\begin{enumerate}
%\item a frame $\varepsilon$ of $L$, 
%\item a symplectic frame 
%\[
%(\delta_1, \ldots, \delta_{k}, \eta^1, \ldots, \eta^{k})
%\]
%of $DL$ for the symplectic structure
%\[
%\varepsilon \circ \omega : \wedge^2 DL \to \bbR_M.
%\]
%\end{enumerate}
%It is easy to see that $(\delta_1, \ldots, \delta_{k}, \varepsilon \otimes \eta^1, \ldots, \varepsilon \otimes \eta^{k})$ identify with a local homogeneous section $\sigma$ of $\mathrm{Fr}(\Ltilde)$ with the following homogeneity property: $\sigma(r\epsilon) = (h_r)_\ast \circ \sigma (\epsilon) \circ A(r)^{-1}$, where 
%\[
%A(r) = \mathrm{diag}(1, \ldots, 1, r, \ldots, r).
%\]
%All such sections span a homogeneous $\mathrm{Sp}_{k}$-structure $S \subset \mathrm{Fr}(\Ltilde)$ such that $\alpha_S$ is given by (\ref{eq:deg_sympl}). This shows that \emph{(i)} is equivalent to \emph{(ii)}.
%
%Next we show that \emph{(ii)} is equivalent to \emph{(iii)}. Begin with $\omega$ as in \emph{(ii)}, 
Setting 
\[
\Theta := i_\bbI \omega, \quad \Upsilon := i_\bbI d_D \omega,
\]
where $\bbI\in\Gamma(DL)$ is the identity operator, we have that $i_\bbI \Theta = i_\bbI \Upsilon = 0$, which implies that $\Theta$ and $\Upsilon$ descend to an $L$-valued 1-form $\theta\in\Omega^1(M; L)$ and an $L$-valued $2$-form $\upsilon \in \Omega^2 (M; L)$ uniquely defined by
\begin{equation}\label{eq:theta,upsilon}
\theta (X_\Delta) = \Theta (\Delta), \quad \upsilon (X_\Delta, X_{\Delta'}) = \Upsilon (\Delta, \Delta'), \hspace{1cm} \forall\; \Delta, \Delta' \in \Gamma (DL).
\end{equation}
Recall from Section \ref{section:geometrylinebundle} that $X_\Delta$ denotes the symbol of the derivation $\Delta$.

%Notice that the homomorphism
%\begin{equation}\label{eq:(1,r)}
%A : \bbR^\times \to N(\mathrm{Sp}_k), \quad r \mapsto \begin{pmatrix}
%	\bbI&\bbO\\
%	\bbO&r\bbI
%	\end{pmatrix}
%\end{equation}
%splits the short exact sequence (\ref{eq:lem:normalizer_symplectic_group:3}), as well

\begin{theorem}\label{theor:cont}
Let $L\to M$ be a line bundle, with $n=\dim M$ odd, and set $k=(n+1)/2$. The assignment $S \mapsto (\theta, \upsilon)$ described above defines a one-to-one correspondence between
\begin{enumerate}[label=(\roman*)]
\item $\alpha$-homogeneous $\mathrm{Sp}_k$-structures $S$ on $\Ltilde$, with $\alpha: \bbR^\times \to N(\mathrm{Sp}_{k})/\mathrm{Sp}_{k} \cong \bbR^\times$ the identity map, 
%\begin{enumerate}
%\item[(i)] a homogeneous $\mathrm{Sp}_{k}$-structure $S$ on $L$ with 
%\begin{equation}\label{eq:deg_sympl}
%\alpha_S(r) =  \mathrm{diag}(1, \ldots, 1, r, \ldots, r)\mathrm{Sp}_{k};
%\end{equation}
%\item[(ii)] an $L$-valued almost symplectic structure on $DL$, i.e.~a non-degenerate $2$-form
%\[
%\omega : \wedge^2 DL \to L;
%\]
\item pairs $(\theta, \upsilon)$ consisting of an $L$-valued one-form $\theta \in \Omega^1 (M; L)$, and an $L$-valued $2$-form $\upsilon \in \Omega^2 (M; L)$ such that
\begin{itemize}
\item $\theta$ is nowhere zero,
\item $\upsilon|_H-R_H$ is a non-degenerate $2$-form on $H:=\ker \theta$, where $R_H$ is the curvature of $H$. 
\end{itemize}
\end{enumerate}
%\end{enumerate}
Furthermore, the following conditions are equivalent:
\begin{enumerate}
\item $S$ is homogeneous integrable,
\item $S$ is integrable,
%\item $\omega$ is symplectic, i.e.~it is a cocycle in the de Rham complex of $DL$ with coefficients in $L$.
\item $\upsilon = 0$, hence $H = \ker \theta$ is a contact structure.
\end{enumerate}
\end{theorem}

\begin{proof}
We first show that the pair $(\theta, \upsilon)$ associated with an $\alpha$-homogeneous $\mathrm{Sp}_{k}$-structure $S$ satisfies the two properties in item \textit{(ii)}. The graded commutator $[d_D,i_\bbI]$ acts like the identity on $\Omega^\bullet(DL,L)$ and it follows that
\[
\omega = d_D i_\bbI \omega + i_\bbI d_D \omega  = d_D \Theta + \Upsilon.
\]
Consequently,
\begin{equation}\label{eq:omega_upsilon_theta}
\omega (\Delta, \Delta') = \Delta (\theta (X_{\Delta'})) - \Delta' (\theta (X_\Delta)) - \theta ([X_{\Delta}, X_{\Delta'}]) + \upsilon (X_{\Delta}, X_{\Delta'}),\hspace{0.5cm}  \forall\; \Delta,\Delta'\in\Gamma(DL).
\end{equation}
Since $\omega$ is non-degenerate and $\bbI$ is nowhere zero, then $\Theta = i_\bbI \omega$, and hence $\theta$, is nowhere zero. Let $H:= \ker \theta \subset TM$ be the induced hyperplane distribution. We want to show that $\upsilon|_H - R_H$ is a non-degenerate, $L$-valued $2$-form on $H$. Recall that the curvature $R_H$ of the distribution $H$ is the $L$-valued $2$-form on $H$, defined by $R_H (X,Y) = [X,Y] \mod \Gamma (H) \in \Gamma (L)$ for all $X,Y \in \Gamma (H)$. Now, pick a connection $\nabla$ on $L$, and note that 
\[
R_H = - d_\nabla \theta |_H
\]
where $d_\nabla$ is the associated connection-differential. So, it is enough to show that the intersection
\[
H \cap \ker \left( \upsilon + d_\nabla \theta \right)
\]
is trivial. The claim will then follow from the fact that, in this case, $\upsilon + d_\nabla \theta $ can only have rank $1$ kernel transversal to $H$, hence $(\upsilon + d_\nabla \theta )|_H = \upsilon|_H - R_H$ must be non-degenerate. So, let $X \in \Gamma (H)$ be such that
\[
\upsilon (X,Y) + d_\nabla \theta(X,Y) = 0
\]
for all $Y \in \mathfrak X (M)$. This means that
\[
0 = \upsilon (X,Y) + \nabla_X (\theta (Y)) - \nabla_Y (\theta (X)) - \theta ([X,Y]) = \omega (\nabla_X, \nabla_Y)
\]
for all $Y$, where we used (\ref{eq:omega_upsilon_theta}). But
\[
\omega (\nabla_X, \bbI) = - \Theta (\nabla_X) = - \theta (X)
\]
vanishes as well. Hence
\[
\omega(\nabla_X, \Delta) = 0
\]
for all $\Delta \in \Gamma (DL)$. As $\omega$ is non-degenerate, we conclude that $X = 0$.

Conversely, let $(\theta, \upsilon)$ be as in \emph{(ii)}, define $\Theta, \Upsilon$ via (\ref{eq:theta,upsilon}) (so that $i_\bbI\Theta=i_\bbI\Upsilon=0)$ and finally put
\[
\omega = d_D \Theta + \Upsilon
\]
(so that $\Theta = i_\bbI \omega$, and $\Upsilon = i_\bbI d_D \omega$). We want to show that $\omega$ is non-degenerate. To do this let $\Delta \in \Gamma (DL)$ be such that
\[
0 = \omega (\Delta, \Delta') = \Delta (\theta (X_{\Delta'})) - \Delta' (\theta (X_\Delta)) - \theta ([X_\Delta, X_{\Delta'}]) + \upsilon (X_\Delta, X_{\Delta'})
\]
for all $\Delta' \in \Gamma (DL)$. In particular,
\[
0 = \omega (\Delta, \bbI) = - \theta (X_\Delta),
\]
showing that $X_\Delta \in \Gamma (H)$. More generally, let us assume that $X_{\Delta'} \in \Gamma (H)$. Then we get
\[
0 = \omega (\Delta, \Delta') = - \theta ([X_\Delta, X_{\Delta'}]) + \upsilon (X_\Delta, X_{\Delta'}) = -R_H (X_\Delta, X_{\Delta'}) + \upsilon (X_\Delta, X_{\Delta'}).
\]
As $X_{\Delta'} \in \Gamma (H)$ is otherwise arbitrary and $\upsilon|_H - R_H$ is non-degenerate, we conclude that $X_\Delta = 0$, so that $\Delta = f \bbI$ for some function $f \in C^\infty (M)$, and
\[
0 = i_{\Delta} \omega = f i_\bbI \omega = f \Theta.
\]
But, from $\theta \neq 0$, it follows that $\Theta \neq 0$ everywhere, hence $f = 0$, i.e.~$\Delta = 0$, showing that $\omega$ is non-degenerate as claimed. This means that, locally, around every point of $M$, we can choose
\begin{enumerate}
\item a basis $\lambda$ of $\Gamma (L)$, and
\item a symplectic frame with components 
\[
\delta_1, \ldots, \delta_{k}, \eta^1, \ldots, \eta^{k}
\]
for the fiber-wise symplectic structure
\[
\lambda^\ast \circ \omega : \wedge^2 DL \to \bbR_M.
\]
\end{enumerate}
where $\lambda^\ast$ is the dual basis of $\lambda$ in $\Gamma (L^\ast)$: $\lambda^\ast (\lambda) = 1$. It is easy to see that the vector fields
\[
\widetilde \delta_1, \ldots, \widetilde \delta_k, \widetilde \lambda{}^{-1} \cdot \widetilde \eta^1, \ldots, \widetilde \lambda{}^{-1} \cdot \widetilde \eta^k 
\]
are the components of a semi-local homogeneous frame $\sigma$ of $\Ltilde$ with the following homogeneity property $\sigma(r\epsilon) = (h_r)_\ast \circ \sigma (\epsilon) \circ A(r)^{-1}$, where $A$ is given in \eqref{eq:SES_Sp_splitting}. All such frames span an $\alpha$-homogeneous $\mathrm{Sp}_{k}$-structure $S$ on $\Ltilde$ with $\alpha$ being the identity, and this construction inverts the assignment $S \mapsto (\theta, \upsilon)$.

For the second part of the statement, that \emph{(1)} implies \emph{(2)} is obvious. Let us show that \emph{(2)} implies \emph{(3)}. So, let $S$ be integrable. Then the associated almost symplectic structure $\widetilde \omega$ is actually a symplectic structure. As the de Rham differential of $\widetilde \omega$ is equal to $\widetilde{d_D \omega}$, we conclude that $\omega$ is $d_D$-closed.
%That
%\emph{(3)} implies \emph{(1)} was already discussed in Section \ref{section:geometrylinebundle}. 
%Alternatively, it is enough to notice that, when $d_D \omega = 0$, then $\Upsilon = i_\bbI d_D \omega = 0$, hence $\upsilon = 0$ as well, and $R_H$ is non-degenerate, so that $H$ is a contact structure.
Then $\Upsilon = i_\bbI d_D \omega = 0$, hence $\upsilon = 0$ as well, and $R_H$ is non-degenerate, so that $H$ is a contact structure.

It remains to show that \emph{(3)} implies \emph{(1)}.  To do this, assume that $S$ is such that $H = \ker \theta$ is a contact structure, and $\upsilon = 0$. Choose Darboux coordinates $(x^i, u, p_i)$ on $M$, so that
\[
\theta = (du - p_i dx^i) \otimes \lambda
\]
where $\lambda = \theta \left( \frac{\partial}{\partial u} \right) \neq 0$ everywhere. It is then easy to see that
\[
\widetilde \omega = d \left( \widetilde \lambda (du -p_i dx^i)\right) = - du \wedge d \widetilde \lambda + dx^i \wedge d (\widetilde \lambda p_i).
\]
This shows that $\chi = (u, x^i, - \widetilde \lambda, \widetilde \lambda p_i)$ are homogeneous coordinates such that $\sigma_\chi$ takes values in $S$. 
\end{proof}

Theorem \ref{theor:cont} shows that, given an integrable (hence homogeneous integrable) $\alpha$-homogeneous $\mathrm{Sp}_{k}$-structure on $\Ltilde$, with $\alpha$ the identity map, we get a contact structure $H$ on $M$ together with an isomorphism $L \cong TM/H$, and vice versa. Thus, contact structures indeed fit in the framework of integrable homogeneous $G$-structures, and this also suggests that (at least from the point of view of $G$-structures) the correct notion of an \textit{almost structure} in contact geometry is a pair $(\vartheta, \upsilon)$ as in the statement of the theorem. 

\begin{remark}
According to Theorem \ref{theor:cont}, it would be natural to call a pair $(H, \upsilon)$ as in the statement an \emph{almost contact structure}. Unfortunately, this terminology is already used in the literature in at least two other situations. The community working on \textit{metric contact geometry} often uses the term \emph{almost contact manifold} to refer to the odd dimensional analogue of an almost complex manifold and/or its metric version (see \cite{Blair}, \cite[Appendix]{Schnitzer} and Section \ref{SecComplex} below). The community working on \textit{contact topology} uses the term \emph{almost contact structure} to denote a pair $(H, \beta)$ (which, in this remark, we will refer to as an \emph{almost contact pair}) consisting of a hyperplane distribution $H \subset M$ and a non-degenerate 2-form $\beta : \wedge^2 H \to L$ on $H$ with values in the normal line bundle $L := TM/H$ (see \cite{Eliashberg}).  The reason for this terminology is that, when $\beta = R_H$, then $H$ is a contact structure. While the latter is very closely related to our pair $(H, \upsilon)$, there is a subtle difference between the two. Namely, if we start with a pair $(H, \upsilon)$ as in Theorem \ref{theor:cont}, then indeed $(H, R_H - \upsilon|_H)$ is an almost contact pair in the above sense. However, if $(H, \beta)$ is an almost contact pair, then $R_H - \beta$ can be extended to a $2$-form $\upsilon \in \Omega^2 (M, L)$ such that $(H, \upsilon)$ is as in Theorem \ref{theor:cont}, but not in a canonical way (one needs to make a choice of a trivialization $L \cong \mathbb R_M$).
\end{remark}

\section{Other Examples}
\label{section:otherexamples}

\subsection{The symplectic group again}

Let $n > 0$ be an odd integer, and set $k=(n+1)/2$. 
%\begin{lemma}
%The normalizer $N(\mathrm{Sp}_{k})$ of the symplectic group $\mathrm{Sp}_{k}$ in $\mathrm{GL}_{n+1}(\bbR)$ is [spanned by $O_{n+1}$ and scalar (invertible) matrices [...].
%\end{lemma}
%
%\begin{proof}
%[...]
%\end{proof}
Let us consider homogeneous $\mathrm{Sp}_{k}$-structures $S$ whose degree $\alpha$ is trivial, i.e.  $\alpha : \bbR^\times \to N(\mathrm{Sp}_{k})/\mathrm{Sp}_{k} \cong \bbR^\times$, with $\alpha (r) = 1$ for all $r \in \bbR^\times$. As we will see, these types of $G$-structures are closely related to cosymplectic structures and arise naturally in $b$-symplectic geometry. 

Recall that a \emph{$b$-manifold} is a pair $(N,M)$ consisting of a manifold $N$ and a closed hypersurface $M \subset N$ (see, e.g.,~\cite{Guillemin}). The \emph{$b$-tangent bundle} of $(N, M)$ is the vector bundle $T^b N$ over $N$ whose sections are vector fields on $N$ that are tangent to $M$. The $b$-tangent bundle has the structure of a Lie algebroid, where the Lie bracket is given by the commutator of vector fields (tangent to $M$) and the anchor map is the identity map at the level of sections. The (point-wise) restriction $T^b N |_M \to M$ is a subalgebroid that fits in the following short exact sequence of vector bundles over $M$:
\begin{equation}\label{eq:SES}
0 \to K \to T^b N |_M \to TM \to 0.
\end{equation}
The projection $T^b N |_M \to TM$ maps (the point-wise restriction to $M$ of) a section of $T^b N$ to its restriction to $M$ as a vector field, and it is well-defined by the definition of $T^b N$. The kernel $K$ admits a canonical nowhere-zero section $I$, which, in a coordinate chart $(t, z^a)$ of $N$ adapted to $M$ (i.e. for which $M$ is the zero set of $t$), is given by $I = t \frac{\partial}{\partial t}$. 

Now, let $\nu := TN|_M /TM \to M$ be the normal bundle to $M$, and let $L := \nu^\ast \to M$ be the conormal bundle. In particular, $L$ is a line bundle, and it is not hard to see that there is an isomorphism of Lie algebroids $T^b N|_M \cong DL$ which maps the point-wise restriction to $M$ of a section $X$ of $T^b N$ to the derivation $\Delta_X$ of $L$ defined as follows: any $\lambda \in \Gamma (L)$ is the point-wise restriction to $M$ of a $1$-form $\eta$ on $N$ whose pull-back to $M$ vanishes, and 
\[
\Delta_X \lambda := (\mathcal L_X \eta)|_M.
\]
Under the isomorphism $T^b N|_M \cong DL$, $I$ becomes the identity derivation $\mathbb I$, and, hence, the short exact sequence (\ref{eq:SES}) becomes
\[
0 \to \bbR_M \to DL \to TM \to 0.
\]

A \emph{\emph{$b$}-symplectic structure} on a $b$-manifold $(N, M)$ is a symplectic structure on the $b$-tangent bundle. $b$-symplectic structures are important in Poisson geometry since they provide particularly nice instances of Poisson manifolds, namely Poisson manifolds $(N, \pi)$ whose Poisson tensor $\pi$ is everywhere non-degenerate except for a hypersurface $M \subset N$, where $\pi$ satisfies a suitable transversality condition. Given a $b$-symplectic structure on $(N,M)$, the restriction of the symplectic form to $T^b N|_M$ can be seen as a symplectic structure on the Atiyah algebroid $DL$ under the isomorphism $T^b N|_M \cong DL$:
\[
\omega : \wedge^2 DL \to \bbR_M.
\]
Such symplectic structures, as Theorem \ref{theor:sympl_triv} below shows, are examples of $\alpha$-homogeneous $\mathrm{Sp}_{k}$-structures with $\alpha = 1$. Let us explain how the symplectic form is constructed from such a structure. 

Let $L \to M$ be a line bundle with $\dim M = n$ odd, and set $k = (n+1)/2$. Let $S$ be an $\alpha$-homogeneous $\mathrm{Sp}_{k}$-structure on $\Ltilde$, with $\alpha = 1$. By Proposition \ref{prop:any_A}, around any point in $\Ltilde$, there is a saturated open neighborhood $\Utilde$ and a homogeneous section $\sigma$ of $S|_{\Utilde}$ such that $A_\sigma = 1$. The components $X_1, \ldots, X_{k}, Y^1, \ldots, Y^{k}$ of $\sigma$ satisfy the homogeneity conditions
\[
(h_r)_\ast (X_i) = X_i \hspace{0.5cm} \text{and} \hspace{0.5cm} (h_r)_\ast (Y^i) = Y^i \hspace{1cm} \forall \; r \in \bbR^\times,\;  i = 1,\ldots, k.
\]
Denoting the components of the dual coframe by $\xi^1, \ldots, \xi^{k}, \eta_1, \ldots, \eta_{k} \in \Omega^1 (\Utilde)$, we define an almost symplectic structure $\widetilde \omega$ on $\Ltilde$ by setting
\[
\widetilde \omega|_{\Utilde} = \xi^1 \wedge \eta_1 + \cdots + \xi^{k} \wedge \eta_{k}, 
\]
for any saturated open neighborhood $\Utilde$ as above. It follows that $\widetilde \omega$ satisfies the homogeneity property 
\[
h^\ast_r (\widetilde \omega) = \widetilde \omega, \quad \forall\; r \in \bbR^\times.
\]
Equivalently, $\widetilde \omega$ maps homogeneous vector fields of degree $0$ to homogeneous functions of degree $0$, i.e.~fiber-wise constant functions, and hence it defines a non-degenerate $2$-form
$
\omega : \wedge^2 DL \to \bbR_M.
$

\begin{theorem}\label{theor:sympl_triv}
Let $L \to M$ be a line bundle, with $n = \dim M$ odd, and set $k = (n+1)/2$. The assignment $S \mapsto \omega$ described above defines a one-to-one correspondence between $\alpha$-homogeneous $\mathrm{Sp}_{k}$-structures $S$ on $\Ltilde$, with $\alpha=1$, and non-degenerate $2$-forms
\[
\omega : \wedge^2 DL \to \bbR_M.
\]
Furthermore, the following conditions are equivalent:
\begin{enumerate}
\item $S$ is homogeneous integrable,
\item $S$ is integrable,
\item $\omega$ is a cocycle in the de Rham complex of $DL$ (with trivial coefficients).
\end{enumerate}
\end{theorem}

\begin{proof}
Begin with a non-degenerate $2$-form $\omega : \wedge^2 DL \to \bbR_M$. Locally, around every point of $M$, we can choose a symplectic frame of $DL$, with components
\[
\delta_1, \ldots, \delta_{k}, \zeta^1, \ldots, \zeta^{k}
\]
and
\[
\widetilde\delta_1, \ldots, \widetilde\delta_{k}, \widetilde\zeta^1, \ldots, \widetilde\zeta^{k} \in \mathfrak X (\Ltilde)
\]
are the components of a semi-local homogeneous frame $\sigma$ on $\Ltilde$ with the following homogeneity property: $\sigma(r\epsilon) = (h_r)_\ast \circ \sigma (\epsilon)$. All such frames span an $\alpha$-homogeneous $\mathrm{Sp}_{k}$-structure $S \subset \mathrm{Fr}(\Ltilde)$ with $\alpha = 1$, and this construction inverts the correspondence $S \mapsto \omega$.

For the second part of the statement, that \emph{(1)} implies \emph{(2)} is obvious.
Let us show that \emph{(2)} implies \emph{(3)}.
So, let $S$ be integrable. Then the associated almost symplectic structure $\widetilde\omega$ is actually a symplectic structure. Similarly as in the previous section the de Rham differential of $\widetilde \omega$ is equal to $\widetilde{d_D \omega}$ (where $d_D$ is the de Rham differential of the Atiyah algebroid $DL$). We conclude that $d_D \omega = 0$. It remains to show that \emph{(3)} implies \emph{(1)}.
To do this, assume that $S$ is such that $d_D \omega = 0$. Then $d \widetilde \omega = 0$, i.e.~$\widetilde \omega$ is a symplectic structure.
Additionally, it follows from the homogeneity condition $h_r^\ast (\widetilde \omega) = \widetilde \omega$, that the Euler vector field $\calE \in \mathfrak X (\Ltilde)$ is an infinitesimal symplectomorphism.
The Carath\'eodory Theorem then states that, around every point in $\Ltilde$, there is a Darboux chart $(V, \chi)$ for $\widetilde{\omega}$ such that $\frac{\partial}{\partial \chi^1} = \calE$.
Integrating the commutation relations 
\[
\left[\calE, \frac{\partial}{\partial \chi^i}\right] = \left[\frac{\partial}{\partial \chi^1}, \frac{\partial}{\partial \chi^i}\right]= 0, \quad \forall i = 2, \ldots, n+1,
\]
 we easily see that $(V, \chi)$ is a homogeneous chart (with $A : \bbR^\times \to \mathrm{GL}_{n+1}(\bbR)$ being the trivial homomorphism).
%The proof is formally identical to that of Theorem \ref{thm:complex} and we leave the details to the reader. We only mentioned that the proof that \emph{(3)} implies \emph{(1)} in the second part of the statement is an adaptation of the proof of~\cite[Theorem A.1.1]{SV2019} based on the fact that, given a no-where zero symplectic vector field $X$ on a symplectic manifold, locally $X$ is the first coordinate vector field of a Darboux chart.
\end{proof}

\begin{remark}
A \emph{cosymplectic structure} \cite{CDYcosymplectic} on a $(2k-1)$-dimensional manifold $M$ is a pair $(\Omega, \eta)$ consisting of a $2$-form and a $1$-form on $M$ such that $\eta \wedge \Omega^{k-1}$ is a volume form and $d\Omega = d\eta = 0$. Intuitively, $\alpha$-homogeneous $\mathrm{Sp}_k$-structures with $\alpha = 1$ are ``intrinsic versions'' of cosymplectic structures, in the same way that contact structures are ``intrinsic versions'' of contact forms, where the latter are obtained from the former when a certain line bundle is equipped with a preferred trivialization. Namely, let $L \to M$ be a line bundle over $M$, and let $\omega : \wedge^2 DL \to \mathbb R_M$ be a closed $2$-form. It is easy to see that, when $L = \mathbb R_M$ is the trivial line bundle, then $\omega$ identifies canonically with a pair $(\Omega, \eta)$ consisting of an honest $2$-form $\Omega \in \Omega^2 (M)$ and an honest $1$-form $\eta \in \Omega^1 (M)$ on $M$. Now, $\omega$ being non-degenerate implies that $\eta \wedge \Omega^{k-1}$ is a volume form, and $\omega$ being closed implies that $d\Omega = d\eta = 0$ (and vice-versa). In other words $(\Omega,\eta)$ is a cosymplectic structure. For instance, when $L \to M$ is the conormal bundle to the distinguished hypersurface of a $b$-symplectic manifold $N$, $\omega$ the restriction of the symplectic structure to $T^b N|_M \cong DL$ as above, and $M$ is co-orientable (and hence $L$ trivializable), then $\omega$ is responsible for the cosymplectic structure on $M$ associated to a trivialization of $L$ \cite[Proposition 10]{Guillemin}. 
\end{remark}

\subsection{The complex group}\label{SecComplex}

Let $n > 0$ be an odd integer, and set $k = (n+1)/2$. Let us now consider homogeneous $\mathrm{GL}_{k}(\bbC)$-structures, where $\mathrm{GL}_{k}(\bbC)$ is the group of invertible $k\times k$ complex matrices embedded as the subgroup of $\mathrm{GL}_{n+1}(\bbR)$ consisting of matrices $A$ satisfying $A^{-1}JA = J$, with $J$ given by (\ref{eq:i_matrix}). 

\begin{lemma}\label{lem:N(GL(C))}
%The normalizer $N(\mathrm{GL}_{k}(\bbC))$ of the complex group $\mathrm{GL}_{k}(\bbC)$ in $\mathrm{GL}_{n+1}(\bbR)$ is [...].
%The normalizer $N(\mathrm{GL}_{k}(\bbC))$ of the general complex group $\mathrm{GL}_{k}(\bbC)$ in $\mathrm{GL}_{n+1}(\bbR)$ is the simidirect product of $\mathrm{GL}_{k}(\bbC)$ itself and the two elements subgroup $Z$ consisting of the identity matrix $\bbI_{n+1}$ and the matrix
%\[
%V = \begin{pmatrix}
%	\bbO&\bbI \\
%	\bbI & \bbO
%	\end{pmatrix}.
%\]
%In other words $N(\mathrm{GL}_{k}(\bbC))$ is the disjoint union of $\mathrm{GL}_{k}(\bbC)$ and $V \cdot \mathrm{GL}_{k}(\bbC)$.
The normalizer $N(\mathrm{GL}_{k}(\bbC))$ of the complex general linear group $\mathrm{GL}_k (\bbC)$ in $\mathrm{GL}_{n+1} (\bbR)$ fits in the split short exact sequence of Lie group homomorphisms
\begin{equation}
	\label{eq:SES_compl}
	\begin{tikzcd}
	1 \arrow[r]&\mathrm{GL}_{k}(\bbC)\arrow[r]&N(\mathrm{GL}_{k}(\bbC))\arrow[r, %swap,
	 "p"]&\bbZ_2\arrow[r]\arrow[l, bend right=-40, %swap, 
	 "\Sigma"]&1
	\end{tikzcd},
	\end{equation}
with $p : N(\mathrm{GL}_k (\bbC)) \to \bbZ_2$ defined by $(-)^{p(B)} \bbI_{n+1} = J^{-1}B^{-1}J B$. Furthermore, a splitting $\Sigma$ is given by $\Sigma (\overline 0) = \bbI_{n+1}$ and
\begin{equation}%\label{eq:Sigma(1)}
\Sigma (\overline 1) = V :=
\begin{pmatrix}
	\bbO&\bbI\\
	\bbI&\bbO
	\end{pmatrix},
\end{equation}
and therefore $N(\mathrm{GL}_k (\bbC))$ decomposes as the semidirect product of $\mathrm{GL}_k (\bbC)$ and the two element subgroup $\{\bbI_{n+1}, V \}$.
\end{lemma} 

\begin{proof}
We use a similar strategy as that of~\cite{Sirola}. Begin noticing that the matrix $V$ in the statement is indeed in the normalizer. Now let $B \in \mathrm{GL}_{n+1}(\bbR)$. It is easy to see that $B$ is in $N(\mathrm{GL}_k (\bbC))$ iff $J^{-1}B^{-1}JB$ is in the centralizer of $\mathrm{GL}_k (\bbC)$. In its turn, the centralizer consists of the scalar multiplications of vectors in $\bbR^{n+1} = \bbC^k$ by invertible complex numbers, i.e.~matrices $C$ of the form
\[
C = a \bbI_{n+1} + bJ, \quad a+ib \in \bbC \setminus \{0\},
\]
as one can easily show using that elements of the centralizer commute with matrices of the form
\[
\begin{pmatrix}
	U&\bbO\\\bbO&U
	\end{pmatrix}
	\quad \text{and} \quad
	 \begin{pmatrix}
	\bbO&U\\
	-U&\bbO
	\end{pmatrix}
\]
with $U\in\mathrm{GL}_{k}(\bbR)$. A direct computation then reveals that $J^{-1}B^{-1}JB$ is $\pm \bbI_{n+1}$. As $J^{-1}V^{-1}JV = -\bbI_{n+1}$ this concludes the proof.
\end{proof}

The surjective homomorphism $p$ in Lemma~\ref{lem:N(GL(C))} induces an isomorphism of groups $N(\mathrm{GL}_{k}(\bbC))/\mathrm{GL}_k (\bbC) \cong \bbZ_2$. There are, therefore, only two Lie group homomorphisms $\alpha : \bbR^\times \to N(\mathrm{GL}_{k}(\bbC))/\mathrm{GL}_k (\bbC) $, the trivial one and the sign. We restrict our attention to the trivial case. The other case is similar and is left to the reader.

Let $L \to M$ be a line bundle with $\dim M = n$, and let $S$ be an $\alpha$-homogeneous $\mathrm{GL}_k(\bbC)$-structures on $\Ltilde$ with $\alpha=1$ (the trivial map). As in the previous example, around any point in $\Ltilde$ there is a saturated open neighborhood $\Utilde$ and a homogeneous section of $S|_{\Utilde}$ such that $A_\sigma = 1$. The components $X_1, \ldots, X_{k}, Y_1, \ldots, Y_{k}$ of $\sigma$ satisfy
\[
(h_r)_\ast (X_i) = X_i \hspace{0.5cm} \text{and} \hspace{0.5cm} (h_r)_\ast (Y^i) = Y^i \hspace{1cm} \forall \; r \in \bbR^\times,\;  i = 1,\ldots, k.
\]
Denote by $\xi^1, \ldots, \xi^{k}, \eta^1, \ldots, \eta^{k}$ the components of the dual coframe, and define a complex structure $\widetilde j$ by setting
\[
{\widetilde j }\, |_{\Utilde}= \xi^1 \otimes Y_1 + \cdots + \xi^{k} \otimes Y_{k} - \eta^1 \otimes X_1 - \cdots - \eta^{k} \otimes X_{k}.
\]
Clearly,
\[
h^\ast_r (\widetilde j) = \widetilde j, \quad \forall \; r \in \bbR^\times.
\]
Equivalently, $\widetilde j$ maps homogeneous vector fields of degree $0$ to themselves, and hence it defines a fiber-wise complex structure 
$
K : DL \to DL.
$

\begin{theorem}
\label{theor:complex}
Let $L \to M$ be a line bundle, with $n = \dim M$ odd, and set $k = (n+1)/2$. The assignment $S \mapsto K$ defines a one-to-one correspondence between 
\begin{enumerate}
	\item[(i)]$\alpha$-homogeneous $\mathrm{GL}_{k}(\bbC)$-structures on $\Ltilde$, with $\alpha = 1$,
	\item[(ii)] fiber-wise complex structures $K$ on $DL$. 
\end{enumerate}
Furthermore, the following conditions are equivalent:
\begin{enumerate}
\item $S$ is homogeneous integrable,
\item $S$ is integrable,
\item $K$ is a complex structure, in the sense that the (Lie-algebroid) Nijenhuis torsion of $K$ vanishes identically.
\end{enumerate}
\end{theorem}

\begin{proof}
Begin with a fiber-wise complex structure $K : DL \to DL$. Locally, around every point of $M$, we can choose a complex frame  of $DL$, with components 
$
\delta_1, \ldots, \delta_{k}, \zeta_1, \ldots, \zeta_{k},
$
 and  
$
\widetilde \delta_1, \ldots, \widetilde \delta_{k}, \widetilde \zeta_1, \ldots, \widetilde \zeta_{k}
$ 
are the components of a semi-local homogeneous frame $\sigma$ on $\Ltilde$ such that: $\sigma(r\epsilon) = (h_r)_\ast \circ \sigma (\epsilon)$. All such frames span an $\alpha$-homogeneous $\mathrm{GL}_{k}(\bbC)$-structure $S \subset \mathrm{Fr}(\Ltilde)$ with $\alpha = 1$, and this construction inverts the correspondence $S \mapsto K$.

For the second part of the statement, that \emph{(1)} implies \emph{(2)} is obvious. Let us show that \emph{(2)} implies \emph{(3)}. So, let $S$ be integrable. Then the associated almost complex structure $\widetilde j$ is actually a complex structure. As the Nijenhuis torsion of $\widetilde j$ vanishes iff so does the Nijenhuis torsion of $K$ (see \cite[Example 2.3.4]{VitaglianoWade}), we conclude that $K$ is a complex structure on the Atiyah algebroid $DL$. It remains to show that \emph{(3)} implies \emph{(1)}, but this is essentially contained in the proof of~\cite[Theorem A.1.1]{Schnitzer}.
\end{proof}

\begin{remark}
A fiber-wise complex structure $K$ on $DL$ is essentially the same as an \emph{almost contact structure} on $M$ in the sense of \cite{Blair} (see \cite[Appendix]{Schnitzer}), and $K$ is integrable iff the associated almost contact structure is \emph{normal} (see \cite{Blair}). Thus, (normal) almost contact structures in the sense of \cite{Blair} fit well in our setting.
\end{remark}

\subsection{The orthogonal group}\label{sec:orthogonal}

We conclude this paper with the Riemannian case of homogeneous $\mathrm O_{n+1}$-structures, where $\mathrm O_{n+1}\subset \mathrm{GL}_{n+1}(\bbR)$ is the orthogonal group. 

\begin{lemma}[{\cite[Theorems 1.10 and 2.9]{Sirola}}]
\label{lem:N(O)}
The normalizer $N(\mathrm{O}_{n+1})$ of the orthogonal group $\mathrm O_{n+1}$ in $\mathrm{GL}_{n+1}(\bbR)$ fits in the split short exact sequence of Lie group homomorphisms
\begin{equation}
	\label{eq:SES_ort}
	\begin{tikzcd}
	1 \arrow[r]&\mathrm{O}_{n+1}\arrow[r]&N(\mathrm{O}_{n+1})\arrow[r, %swap,
	 "p"]&\bbR_{+}^\times \arrow[r]\arrow[l, bend right=-40, %swap, 
	 "\Sigma"]&1
	\end{tikzcd},
	\end{equation}
with $\bbR^\times_+$ the multiplicative group of positive reals, and $p : N(\mathrm{O}_{n+1}) \to \bbR^\times_+$ defined by $p(B)\bbI_{n+1} = B^t B$. Furthermore, a splitting $\Sigma$ is given by $\Sigma (r) = r^{1/2} \bbI_{n+1}$, and therefore $N(\mathrm{O}_{n+1})$ decomposes as the semidirect product of $\mathrm{O}_{n+1}$ and the $1$-dimensional subgroup consisting of positive scalar matrices.
%is spanned by $\mathrm O_{n+1}$ and scalar (invertible) matrices.
%The normalizer $N(\mathrm{O}_{n+1})$ of the orthogonal group $\mathrm O_{n+1}$ in $\mathrm{GL}_{n+1}(\bbR)$ is the semi-direct product of $\mathrm O_{n+1}$ and $\bbR_{>0}\bbI_{n+1}$, namely the subgroup formed by the positive multiple of the unit matrix.
\end{lemma}

The surjective homomorphism $p$ in Lemma \ref{lem:N(O)} induces an isomorphism $N(\mathrm{O}_{n+1})/\mathrm{O}_{n+1} \cong \bbR^\times_+$. In this final example, we consider $\alpha$-homogeneous $\mathrm O_{n+1}$-structures with $\alpha: \bbR^\times \to N(\mathrm{O}_{n+1})/\mathrm{O}_{n+1} \cong \bbR^\times_+$ the square root of the absolute value, i.e. $\alpha (r) = |r|^{1/2}$. The other cases are similar and are left to the reader.

Let $L \to M$ be a line bundle with $\dim M = n$, and let $S$ be an $\alpha$-homogeneous $\mathrm O_{n+1}$-structure with $\alpha$ as above. While in the case of usual $G$-structures, an $\mathrm O_{n}$-structure on a manifold is encoded by a metric on that manifold, an $\alpha$-homogeneous $\mathrm O_{n+1}$-structure on $\Ltilde$ gives rise to (and is encoded by, as we will prove) a triple $(\phi, g, \eta)$ consisting of an orientation preserving trivialization $\phi : |L| \cong \bbR_M$ of the line bundle $|L|$, a Riemannian metric $g$ on the base manifold $M$, and a $1$-form $\eta \in \Omega^1 (M)$. Let us explain how such a triple is constructed.

Around any point in  $\Ltilde$ there is a saturated open neighborhood $\Utilde$ and a homogeneous section $\sigma$ of $S|_{\Utilde}$ such that $A_\sigma (r) = |r|^{1/2} \bbI_{n+1}$ for all $r \in \bbR^\times$. The components $X_1, \ldots, X_{n+1}$ of $\sigma$ satisfy
\[
(h_r)_\ast (X_i) = |r|^{1/2} X_i \hspace{1cm} \forall \; r \in \bbR^\times,\;  i = 1,\ldots, n+1.
\]
Denote by $\xi^1, \ldots, \xi^{n+1}$ the components of the dual coframe, and define a Riemannian metric $\widetilde g$ by setting
\[
\widetilde g|_{\Utilde} = \xi^1 \odot \xi^1 + \cdots + \xi^{n+1} \odot \xi^{n+1}.
\]
This metric satisfies the homogeneity property
\[
h^\ast_r (\widetilde g) = |r| \widetilde g, \quad \forall\; r \in \bbR^\times,
\]
or, equivalently, $\widetilde g$ maps a pair of homogeneous vector fields of degree $0$, say $X,Y$, to a function $\widetilde g (X,Y)$ such that $h_r^\ast (\widetilde g (X,Y)) = |r| \widetilde g (X,Y)$. This, in turn, implies that $\widetilde g$ defines a definite, symmetric bilinear form (see Remark \ref{rem:phi-hom})
\[ 
G : DL \odot DL \to |L|.
\]
%
%{
%\color{blue}
%Let $f : \bbR^\times \to \bbR^\times$ be a Lie group homomorphism, and let $L \to M$ be a line bundle. It is easy to construct a new line bundle $f(L) \to M$ with this data as follows. As usual, let $\Ltilde$ be the frame bundle of $L$, and recall that $L$ is the vector bundle associated to the tautological action of $\bbR^\times = \mathrm{GL}_1(\bbR)$ on $\bbR$. The homomorphism $f$ defines a new action $\bbR^\times \to \mathrm{GL}_1(\bbR) = \bbR^\times$ and we call $f(L)$ the associated line bundle. In other words, a point in $f(L)$ over a point $x \in M$ is a class $[(\epsilon, t)]$, where $\epsilon$ is a frame in $L_x$, and $t$ is a real number. Two pairs $(\epsilon, t), (\epsilon', t')$ are in the same class if there exists a non-zero real number $r$ such that $\epsilon' = r \epsilon$ and $t' = f(r) t$.
%}

Now, since there is a canonical isomorphism of Lie algebroids $D|L| \cong DL$ (see again Remark \ref{rem:phi-hom}), $G$ is the same as a definite, symmetric $|L|$-valued bilinear form on $D|L|$, which we also denote by $G$. In particular, $ G(\bbI, \bbI)$ is a non-zero section of $|L|$ and it induces an orientation preserving trivialization 
\[
\phi : |L| \cong \bbR_M.
\] 
We can, therefore, identify $|L|$ with the trivial line bundle $\bbR_M$. Next, the $G$-orthogonal bundle $\bbI^\bot \subset D |L|$ is the image of a unique linear connection $\nabla^{|L|}: TM \to D|L|$ on $|L|$, and, since $|L|$ is a trivial line bundle, the connection $\nabla^{|L|}$ defines a connection $1$-form $\eta$ on $M$ via 
\[
\eta (X) = (\phi \circ \nabla^{|L|}_X \circ \phi^{-1}) (1),
\] 
for all $X \in \mathfrak X (M)$. Finally, we can also use $\nabla^{|L|}$ to identify $TM$ and $\bbI^\bot$, which allows us to regard the restriction of $G$ to $\bbI^\bot$ as the Riemannian metric $g$ on $M$ defined by
\[
g(X,Y) = (\phi \circ G)(\nabla^{|L|}_X, \nabla^{|L|}_Y).
\]

Moving on to the question of integrability, we know that since an $\alpha$-homogeneous $\mathrm O_{n+1}$-structure on $\Ltilde$ is, in particular, a Riemannian structure, then it is integrable if and only if the curvature of the induced metric $\widetilde{g}$ vanishes. In this setting, however, this can be stated more elegantly as the vanishing of the curvature of $G$, in the sense of Lie algebroids. Let us recall this notion of curvature.

The \emph{Fundamental Theorem of Riemannian Geometry} (for Lie algebroids) says that there exists a unique $D|L|$-connection $\nabla^D$ in $D|L|$ such that 
\begin{enumerate}
\item $\nabla^D_{\Delta_1} \Delta_2 - \nabla^D_{\Delta_2} \Delta_1 = [\Delta_1, \Delta_2]$ (i.e.~$\nabla^D$ is a \emph{symmetric connection}), 
\item $\Delta_1 (G (\Delta_2, \Delta_3)) = G (\nabla^D_{\Delta_1} \Delta_2, \Delta_3) + G (\Delta_2, \nabla^D_{\Delta_1} \Delta_3)$ (i.e.~$\nabla^D$ is a \emph{metric connection}), 
\end{enumerate}
for all $\Delta_1, \Delta_2, \Delta_3 \in \Gamma (D|L|)$. The curvature of $G$ is then defined as the $2$-form
\begin{equation}\label{eq:RD}
R^D : \wedge^2 D|L| \to \operatorname{End} D|L|, \quad (\Delta_1, \Delta_2) \mapsto [\nabla^D_{\Delta_1}, \nabla^D_{\Delta_2}] - \nabla^D_{[\Delta_1, \Delta_2]}.
\end{equation}

Since $G$ can be encoded in terms of the triple $(\phi, g, \eta)$, we should be able to express $R^D$ solely in terms of this data. Indeed, the curvature $R^D$ is determined by the following formulae (the computation is straightforward but lengthy, and we suffice with presenting here only the final result):
\[
\begin{aligned}
R^D (\mathbb I, \nabla^{|L|}_X) \mathbb I & = \tfrac{1}{4}\nabla^{|L|}_{A(X)}, \\
R^D (\mathbb I, \nabla^{|L|}_X) \nabla^{|L|}_Y & = \tfrac{1}{4} \left(\nabla^{|L|}_{B(X,Y)} -  g(A(X),Y) \mathbb I\right), \\
R^D (\nabla^{|L|}_X, \nabla^{|L|}_Y) \mathbb I & = -\tfrac{1}{4}  \nabla^{|L|}_{C(X,Y)}, \\
R^D (\nabla^{|L|}_X, \nabla^{|L|}_Y) \nabla^{|L|}_Z & = \tfrac{1}{4}\left(\nabla^{|L|}_{D(X,Y, Z)} - g\left(C(X,Y), Z\right) \mathbb I \right),
\end{aligned}
\] 
where $A,B, C$ and $D$ are the tensors defined by
\begin{equation}\label{eq:A,B}
\begin{aligned}
g(A(X), Y) & = i_Y\nabla_X \eta + i_X \nabla_Y \eta - \eta (X) \eta(Y)  + g( \eta^\sharp, \eta^\sharp)g (X,Y) - g ((i_X d\eta)^\sharp, (i_Y d\eta)^\sharp), \\
B(X,Y)_\flat & = 2 i_Y \nabla_X d\eta - \left( \eta (Y) + d\eta (\eta^\sharp, Y)\right) X_\flat + g(X,Y) \left( \eta + i_{\eta^\sharp} d\eta \right), \\
C(X,Y) & = B(X,Y) - B(Y,X) 
\end{aligned}
\end{equation}
and
\begin{equation}\label{eq:D}
\begin{aligned}
g(D(X,Y,Z),W) 
& = 2 g(R(X,Y,Z),W) + g (X,Z) g(Y,W) \\
& \quad + \left(i_Z\nabla_X \eta +  i_X\nabla_Z \eta - \eta(X)\eta(Z) + g(\eta^\sharp, \eta^\sharp)g(X,Z) \right) g(Y,W) \\
& \quad - \left(i_W \nabla_X \eta + i_X \nabla_W \eta - \eta(X) \eta(W)\right) g(Y,Z) \\
& \quad + d\eta (X,Y) d\eta(W,Z) + d\eta (X,Z) d\eta(Y,W) - (X \leftrightarrow Y).
\end{aligned}
\end{equation}
for all $X, Y, Z, W \in \mathfrak X (M)$. Here $\sharp : T^\ast M \to TM$ and $\flat = \sharp^{-1} : TM \to T^\ast M$ are the musical isomorphisms, $\nabla$ is the Levi-Civita connection, $R$ is the Riemann tensor of $g$, and $(X \leftrightarrow Y)$ denotes swapping $X$ and $Y$.

\begin{theorem}
\label{theor:riemannian}
Let $L \to M$ be a line bundle, with $n = \dim M$. The assignments $S \mapsto G \mapsto (\phi, g, \eta)$ establish one-to-one correspondences between
\begin{enumerate}
\item[(i)]$\alpha$-homogeneous $\mathrm O_{n+1}$-structures $S$ on $\Ltilde$, with $\alpha$ being the square root of the absolute value,
\item[(ii)]definite, symmetric, bilinear forms $G : DL \odot DL \to |L|$,
\item[(iii)] triples $(\phi, g, \eta)$ consisting of an orientation preserving trivialization $\phi : |L| \cong \bbR_M$ of the line bundle $|L|$, a Riemannian metric $g$ on $M$, and a $1$-form $\eta \in \Omega^1 (M)$.
\end{enumerate}
Furthermore, the following conditions are equivalent:
\begin{enumerate}
	\item $S$ is integrable,
	\item the curvature $R^D$ of $G$, defined in \eqref{eq:RD}, vanishes,
	\item the tensors $A, B$ and $D$, constructed from $g$ and $\eta$ as in (\ref{eq:A,B}) and (\ref{eq:D}), vanish.
\end{enumerate}
\end{theorem}

\begin{proof}
Let $(\phi, g, \eta)$ be a triple as in the statement. We can use $\phi$ to identify $|L|$ and $\bbR_M$, and $\eta$ with a linear connection $\nabla^{|L|} : TM \to D|L|$. Since $\nabla^{|L|}$ is an isomorphism onto its image $H = \nabla^{|L|}(TM)$, $g$ defines a fiber-wise scalar product on $H$ and can be uniquely extended to a fiber-wise scalar product $G$ on $D|L|$ such that $G(\bbI, \bbI) = 1$ and $H = \bbI^\bot$. Identifying $|L|$ with the trivial line bundle again, and $D |L|$ with $DL$, we can regard $G$ as a definite, symmetric bilinear form
$
G : DL \odot DL \to |L|.
$
Additionally, $u = G(\mathbb I, \mathbb I)$ is a positively oriented basis of $\Gamma (|L|)$. Locally, around every point of $M$, we can choose an orthonormal frame with components
$
\delta_1, \ldots, \delta_{n+1}
$
for the fiber-wise scalar product $u^\ast \circ G : DL \odot DL \to \bbR_M$, where $u^\ast$ is the dual basis of $u$ in $\Gamma (|L|^{-1})$. It is easy to see that the vector fields
$
\widetilde u{}^{-1/2} \cdot \widetilde \delta_1, \ldots, \widetilde u{}^{-1/2} \cdot \widetilde \delta_{n+1}
$
are well-defined and that they are the components of a semi-local homogeneous frame $\sigma$ of $\Ltilde$ such that $\sigma (r \epsilon) = |r|^{-1/2}(h_r)_\ast \circ \sigma (\epsilon)$. All such frames span an $\alpha$-homogeneous $\mathrm{O}_{n+1}$-structure $S$ on $\Ltilde$ with $\alpha$ being the square-root of the absolute value. This construction inverts the assignment $S \mapsto (\phi, g, \eta)$.

For the second part of the statement, take an integrable $\alpha$-homogeneous $\mathrm{O}_{n+1}$-structure $S$ on $\Ltilde$ with $\alpha$ being the square-root of the absolute value. As above, $S$ determines a Riemannian metric $\widetilde g$ on $\Ltilde$. From integrability, $\widetilde g$ is a flat metric and we want to show how this translates in terms of the data $(\phi, g, \eta)$. To do this, we use $\phi$ to identify $|L|$ with the trivial line bundle. Now, using the fact that $\widetilde g$ determines a definite, symmetric bilinear form $G : D |L| \odot D|L| \to |L|$,
%, and a (Lie algebroid) version of the \emph{Fundamental Theorem of Riemannian Geometry} says that there exists a unique $D|L|$-connection $\nabla^D$ in $D|L|$ such that 
%\begin{enumerate}
%\item $\nabla^D_{\Delta_1} \Delta_2 - \nabla^D_{\Delta_2} \Delta_1 = [\Delta_1, \Delta_2]$ (i.e.~$\nabla^D$ is a \emph{symmetric connection}), and
%\item $\Delta_1 (G (\Delta_2, \Delta_3)) = G (\nabla^D_{\Delta_1} \Delta_2, \Delta_3) + G (\Delta_2, \nabla^D_{\Delta_1} \Delta_3)$ (i.e.~$\nabla^D$ is a \emph{metric connection}), 
%\end{enumerate}
%for all $\Delta_1, \Delta_2, \Delta_3 \in \Gamma (D|L|)$.
the flatness of $\widetilde g$ is equivalent to the vanishing of the curvature $R^D$ of $G$, which, in turn, is equivalent to the vanishing of the tensors $A, B$ and $D$. 
\end{proof}

\begin{remark}
While for usual $\mathrm O_n$-structures integrability implies (and is equivalent to) the vanishing of the curvature of the induced metric, Theorem \ref{theor:riemannian} shows that for homogeneous $\mathrm O_{n+1}$ structures on $\Ltilde$, with the natural choice of the square root function for $\alpha$, integrability implies a certain system of nonlinear PDEs for the induced metric and $1$-form. This is a curious phenomenon, which, if unknown in the Riemannian geometry literature (we we unable to find any mention of these PDEs), would be worth further investigation. 
\end{remark}

\begin{remark}
Unlike in the examples of the symplectic and complex groups, in the case of the orthogonal group we were unable to prove that integrability of the homogeneous $G$-structure implies homogeneous integrability due to the complications inherent to the equations $A = B = D = 0$ for $g$ and $\eta$. However, in the special case when $\eta = 0$, we obtain the following proposition. 
\end{remark}

\begin{proposition}
In the setting of Theorem \ref{theor:riemannian}, if the $\alpha$-homogeneous $\mathrm O_{n+1}$-structure $S$ is such that $\eta=0$, then $S$ is integrable if and only if it is homogeneous integrable. 
\end{proposition}

\begin{proof}
Under the assumptions in the statement, the equations $A = B = 0$ are trivially fulfilled and the condition $D = 0$ boils down to $(M,g)$ being of constant curvature equal to $1/4$. Then, locally around every point, $(M,g)$ is isometric to the $n$-sphere of radius $2$. Locally, we can assume that $g = 4 g (S^n)$, where $g (S^n)$ is the metric of the unit sphere, and we can also express $g(S^n)$ in spherical coordinates $(z^1, \ldots, z^n)$.  Let $\widetilde g$ be as in the proof of Theorem \ref{theor:riemannian}, then:
\begin{enumerate}
\item $\widetilde g$ is equivalent to a definite, symmetric bilinear form $G : DL \odot DL \cong D|L| \odot D |L| \to |L|$,
\item we encode $G$ in the triple $(\phi, g, \eta)$, where $\phi:|L|\cong\bbR_M$ is the trivialization identifying $u = G(\bbI, \bbI)$ with the constant function $1$, and $\eta$ is the connection $1$-form of the unique connection $\nabla^{|L|}$ in $|L|$ whose image is the $G$-orthogonal complement $\bbI^\bot$,
\item if $\eta = 0$, then $\nabla^{|L|}$ is a flat connection and $u$ is a flat section,
\item being a positively oriented non-zero section of $|L|$, $u$ corresponds to a positive smooth function $\widetilde u \in C^\infty (\Ltilde)$ such that $h_r^\ast (\widetilde u) = |r| \widetilde u$, for all $r \in \bbR^\times$.
\end{enumerate}
From these facts, and the concrete relationship between $\widetilde g$ and $G$, it is easy to see that, when $\eta = 0$ and $g = 4g (S^n)$, we have, locally, that
\[
\widetilde g = \widetilde u ( \widetilde u^{-2} d \widetilde u \odot d \widetilde u + 4g (S^n)).
\]
Hence, setting $R = 2\widetilde{u}^{1/2}$, we find that
\[
\widetilde g = dR \odot dR + R^2 g(S^n),
\]
which is a flat metric. Passing from spherical coordinates $(R, z^1, \ldots, z^n)$ to Cartesian coordinates $(\chi^1, \ldots, \chi^{n+1})$, we can put $\widetilde g$ in the normal form
\[
\widetilde g = d\chi^1 \odot d\chi^1 + \cdots + d\chi^{n+1} \odot d\chi^{n+1}.
\]
The Cartesian coordinates are of the form $R \cdot Y^i (z^1, \ldots, z^n)$, for some smooth functions $Y^i$ of the variables $z^i$, and, hence, they are homogeneous in the sense that
\[
h_r^\ast (\chi^i) = h_r^\ast (R) \cdot Y^i (z^1, \ldots, z^n) = |r|^{1/2} \chi^i, \quad \forall \;r \in \bbR^\times, \; i =1, \ldots, n+1.
\]
We conclude that the $\mathrm{O}_{n+1}$-structure consisting of orthonormal frames of $\widetilde g$ is homogeneous integrable, as claimed.
\end{proof}

\end{document}